\documentclass{amsproc}
\usepackage[dvipdfm]{graphicx}
\usepackage{fancybox}
\usepackage[all]{xy}
\title{Veech groups of flat structures on Riemann surfaces}
\date{}

\author{Yoshihiko Shinomiya}
\address{Department of Mathematics
Tokyo Institute of Technology
2-12-1 Ookayama, Meguro-ku, Tokyo 152-8551, JAPAN}
\email{shinomiya.y.aa@m.titech.ac.jp}

\begin{document}
\maketitle
\bibliographystyle{alpha} 

\theoremstyle{plain}
\newtheorem{theorem}{Theorem}[section]

\theoremstyle{definition}
\newtheorem{definition}[theorem]{Definition}

\theoremstyle{plain}
\newtheorem{proposition}[theorem]{Proposition}

\theoremstyle{plain}
\newtheorem{lemma}[theorem]{Lemma}

\theoremstyle{plain}
\newtheorem{corollary}[theorem]{Corollary}

\theoremstyle{definition}
\newtheorem{example}[theorem]{Example}

\theoremstyle{remark}
\newtheorem*{remark}{\bf Remark}

\begin{abstract}

In this paper, we construct new examples of Veech groups by extending Schmith\"usen's method 
for calculating Veech groups of origamis to Veech groups of unramified finite coverings of regular $2n$-gons.
We calculate the Veech groups of certain Abelian coverings of regular $2n$-gons by using an algebraic method.
\end{abstract}

\section{Introduction}
The Teichm\"uller disk is a holomorphic isometric embedding of an upper-half plane $\mathbb{H}$ (or a unit disk) into a Teichm\"uller space.
All such embeddings are constructed by flat structures on Riemann surfaces and ${\rm SL}(2,\mathbb{R})$-orbit on flat structures.
To study the image of a Teichm\"uller disk into the moduli space, we consider the stabilizer of the Teichm\"uller disk in the mapping class group.
Veech \cite{Veech89} showed that this stabilizer is regarded as the group of all affine diffeomorphisms on a corresponding flat structure and its action can be represented by a Fuchsian group which acts on $\mathbb{H}$.
The Fuchsian group is called a Veech group.

The first non-trivial examples of Veech groups were given by Veech \cite{Veech89} and \cite{Veech91}. 
His examples are constructed by gluing two congruent regular polygons along one side and identifying the parallel sides of the resulting polygons. 
However, not so many examples are known other than Veech's. 
Recently, Schmith\"usen \cite{Schmithusen04} showed an algorithm for finding Veech groups of ``origami''. 
 An origami is an unramified finite covering of a once punctured torus constructed by a unit square.
We apply her method to unramified finite coverings of regular $2n$-gons instead of the unit square
to obtain other examples of Veech groups. 
Veech groups of universal coverings play an important role in her method.
We call these groups universal Veech groups. 

In this paper, we determine the universal Veech groups of $2n$-gons and give an algorithm to calculate Veech groups of finite Abelian coverings of $2n$-gons.
In the case of origamis, Schmith\"usen connected the Veech groups of origamis with 
subgroups of ${\rm SL}(2, \mathbb{Z})$.
She showed that the calculations of Veech groups stop in finitely many steps.
In our case, for the Veech groups of Abelian coverings of $2n$-gons whose degree is $d$, we connect them 
with subgroups of  ${\rm SL}(n, \mathbb{Z}_d)$.
We show that the calculations of Veech groups of certain Abelian coverings 
can be done by using the corresponding subgroups of  ${\rm SL}(n, \mathbb{Z}_d)$.

\section{Definitions}
Let $X$ be a Riemann surface of type $(g,n)$ with $3g-3+n>0$. 
\begin{definition}[Holomorphic quadratic differential]
A holomorphic quadratic differential $\varphi$ on $X$ is a tensor whose restriction to every coordinate neighborhood $(U,z)$ is the form $fdz^2$, here $f$ is a holomorphic function on $U$.\\
We define $|\varphi |$ to be the differential 2-form on $X$ whose restriction to every coordinate neighborhood $(U,z)$ has the form $|f|dxdy$ if $\varphi$ equals $fdz^2$ in $U$. We say $\varphi$ is integrable if its norm
\begin{eqnarray}
||\varphi ||=\int\!\!\!\int_X|\varphi| \nonumber 
\end{eqnarray}
is finite.
\end{definition}
We fix an integrable holomorphic quadratic differential $\varphi$. 
Denote by $X^\prime$ the Riemann surface constructed from $X$ by removing zeros of $\varphi$.

\begin{definition}[Flat structure]
A flat structure ${\it u}$ on $X^{\prime}$ is an atlas of $X^{\prime}$ which satisfies the following conditions.
\begin{enumerate} 
\item[(1)] Local coordinates of ${\it u}$ are compatible with the orientation on $X^{\prime}$ induced by its Riemann surface structure.
\item[(2)] For coordinate neighborhoods $(U,z)$ and $(V,w)$ of ${\it u}$ with $U\cap V \not=\phi $, the transition function is the form 
\begin{center}
$w=\pm z+c$ 
\end{center}
in $z(U\cap V)$ for some $c \in \mathbb{C}$.
\item[(3)]  ${\it u}$ is maximal with respect to $(1)$ and $(2)$.
\end{enumerate}
\end{definition}
The holomorphic quadratic differential $\varphi$ determines a flat structure $u_\varphi $ on $X^{\prime}$ as follows.\\
For each $p_0\in X^{\prime}$, we can choose an open neighborhood $U$ such that
\begin{eqnarray}
z(p)=\int_{p_0}^{p}\sqrt\varphi \nonumber
\end{eqnarray} 
is a well-defined and injective function of $U$. This function is holomorphic in $U$ since $\varphi$ is a holomorphic quadratic differential. 
If $(U,z)$ and $(V,w)$ are pairs of such neighborhoods and functions with $ U \cap V\not=\phi$, then we have $dw^2=\varphi=dz^2$ in $U \cap V$. 
Hence $w=\pm z+c$ in $z(U\cap V)$ for some $c \in \mathbb{C}$. 
The flat structure $u_\varphi $ is the maximal flat structure which contains such pairs.

\begin{definition}[Affine group of $\varphi$]
The affine group $Aff^+(X, \varphi)$ of the integrable holomorphic quadratic differential $\varphi$ is the group of all quasiconformal mappings $f$ of $X$ onto itself which satisfy $f(X^\prime)=X^\prime$ and are affine with respect to the flat structure $u_\varphi$. 
This means that for $(U,z), (V,w) \in u_\varphi$ with $f(U) \subseteq V$, the homeomorphism $w\circ f\circ z^{-1}$ is the form $z \mapsto Az+c$ for some $A \in {\rm GL} (2,\mathbb{R})$ and $c \in \mathbb{C}$.
\end{definition}
This $A$ is uniquely determined up to the sign since $u_\varphi$ is a flat structure.
And $A$ is always in ${\rm SL}(2,\mathbb{R})$ since $||\varphi||=\int_X |\varphi|=\int_X f^*(|\varphi|)=\det(A)||\varphi||$. 
Thus we have a group homomorphism

\begin{center}
$D: Aff^+(X,\varphi) \rightarrow {\rm PSL}(2,\mathbb{R})$.
\end{center}

\begin{definition}[Veech group of $\varphi$]
We call $\Gamma(X,\varphi)=D(Aff^+(X,\varphi))$ the Veech group of $\varphi$. 
\end{definition}

\begin{remark}
Veech groups are discrete subgroups of ${\rm PSL}(2,\mathbb{R})$ (see \cite{EarGar97}). 
\end{remark}

\section{Examples of Veech groups}
In this section, we see two examples of Veech groups.
The first example is a new example of Veech groups.
The second one is the main target of this paper.
The purpose of this paper is to determine Veech groups of some coverings of the second one.
To do this, we need to determine the Veech group of the second one.

\begin{example}
Let $X$ be a surface constructed as Figure \ref{1}.
We induce an unique conformal structure on $X$ such that the quadratic differential $dz^2$ on the interior of the rectangle of Figure \ref{1} extends to a holomorphic quadratic differential $\varphi$ on $X$. 
Then $X$ is a Riemann surface of type $(2,0)$ and vertices of four squares become two points on $X$. 
These points are zeros of $\varphi$ of order 2.
We can see that 
{  $\left(
  \begin{array}{cc}
    1 & 1 \\
    0 & 1 
  \end{array}
  \right)$}
 and 
 {  $\left(
  \begin{array}{cccc}
    1 & 0 \\
    2 & 1
  \end{array}
  \right)$}
define elements in $Aff^+(X,\varphi)$ as Figure \ref{2}.
Hence 
$\Gamma=\left< {  \left[\left(
  \begin{array}{cccc}
    1 & 1 \\
    0 & 1 
  \end{array}
  \right)
  \right]}, 
{   \left[\left(
  \begin{array}{cccc}
    1 & 0 \\
    2 & 1
  \end{array}
  \right)
  \right] }
\right>$ is a subgroup of the Veech group $\Gamma(X,\varphi)$. Since every element in $Aff^+(X,\varphi)$ must preserve the set of all lattice points, $\Gamma(X,\varphi)$ is a subgroup of ${\rm PSL}(2,\mathbb{Z})$. It is known that $\left<
{  \left[ \left(\begin{array}{cccc}
    1 & 2 \\
    0 & 1 
  \end{array}
 \right) \right] }, 
 { 
  \left[\left(
  \begin{array}{cccc}
    1 & 0 \\
    2 & 1
  \end{array} \right)
  \right]}
\right>$ is the congruence subgroup of level $2$ and has index $6$ in ${\rm PSL}(2,\mathbb{Z})$. 
Hence $\Gamma(X,\varphi)$ is either $\Gamma$ or ${\rm PSL}(2,\mathbb{Z})$. 
However, 
  ${ 
\left(
  \begin{array}{cccc}
    1 & 0 \\
    1 & 1 
  \end{array}
  \right)}$ 
cannot be an element in  $\Gamma(X,\varphi)$. Therefore $\Gamma(X,\varphi)$ must be $\Gamma$.\\
\begin{figure}[h]
 \begin{center}
  \includegraphics[keepaspectratio, scale=0.45]{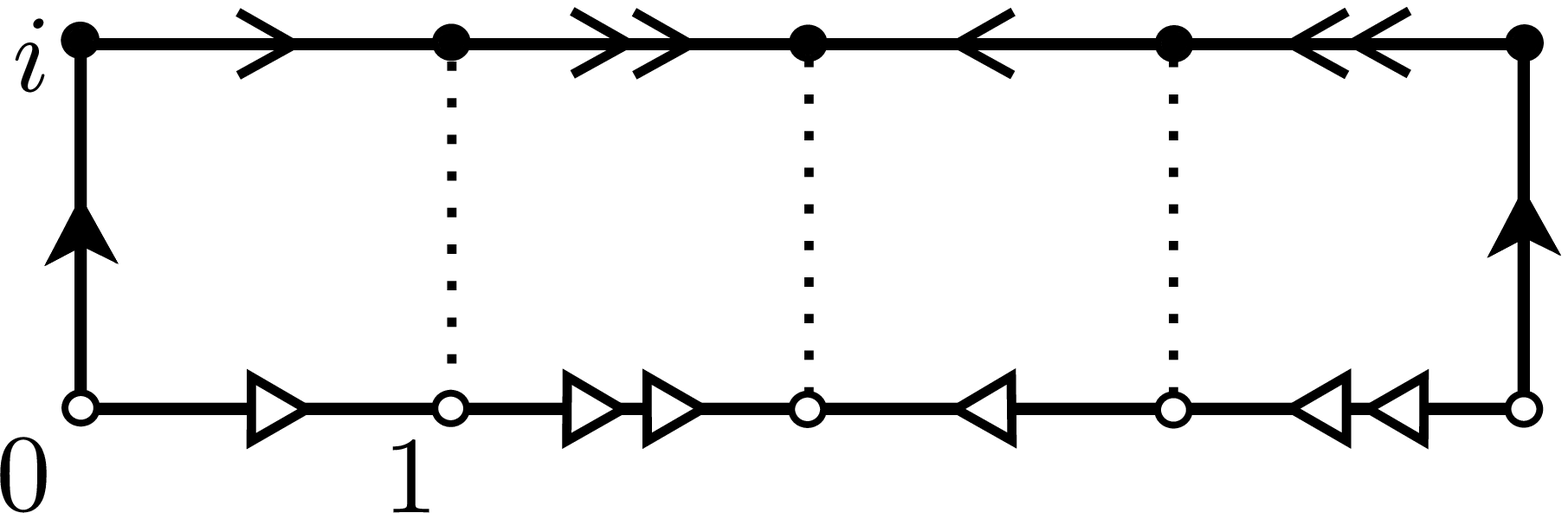}

\caption{}
\label{1}
\end{center}

\end{figure}
 \begin{figure}[h]
 \begin{center}
  \includegraphics*[keepaspectratio, scale=0.45]{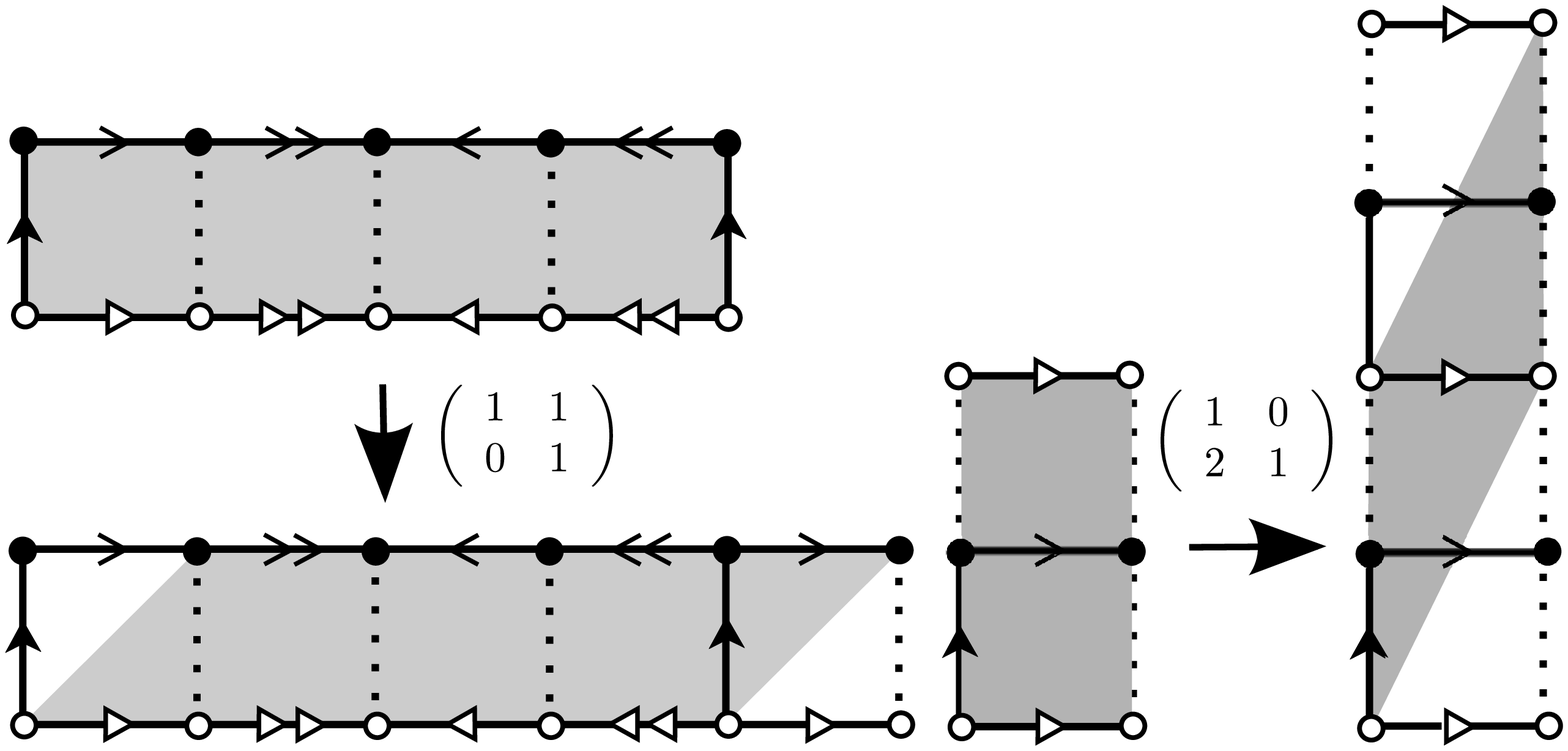}
 \label{Veechexam1-2}
 \caption{}
\label{2}
\end{center}
\end{figure}
\end{example}
The next example is given by Earle and Gardiner (\cite{EarGar97}).
\begin{example}\label{exampleofveech2}
Fix $n\geq 4$ and let $\Pi_{2n}$ be a regular $2n$-gon.
We assume that $\Pi_{2n}$ has two horizontal sides, lengths of the sides are 1 and its vertices are removed. 
We identify each side of  $\Pi_{2n}$ with the opposite parallel side by an Euclidean translation (see Figure \ref{3}) and denote the resulting surface by $P_{2n}$.
We induce an unique conformal structure on $P_{2n}$ such that the quadratic differential $dz^2$ on the interior of $\Pi_{2n}$ extends to a holomorphic quadratic differential $\varphi_{2n}$ on $P_{2n}$. 
 If $n$ is even, then $P_{2n}$ is a Riemann surface of type $(\frac{n}{2}, 1)$ and if $n$ is odd, then $P_{2n}$ is a Riemann surface of type $(\frac{n-1}{2}, 2)$. 
Now
$R_{2n}={  
\left(
  \begin{array}{cccc}
    \cos\frac{\pi}{n} & -\sin\frac{\pi}{n} \\
    \sin\frac{\pi}{n} & \cos\frac{\pi}{n} 
\end{array}
  \right)
}$
and 
$T_{2n}={ 
\left(
  \begin{array}{cccc}
    1 & 2\cot\frac{\pi}{2n} \\
    0 & 1
\end{array}
  \right)}$ 
induce elements in $Aff^+(P_{2n}, \varphi_{2n} )$. 
The action of $R_{2n}$ on $P_{2n}$ is the rotation about the center of $\Pi_{2n}$ of angle $\frac{\pi}{n}$. 
To see the action of $T_{2n}$ on $P_{2n}$, we cut $P_{2n}$ along all horizontal segments which connect the vertices of $\Pi_{2n}$. 
If $n$ is even, $P_{2n}$ is decomposed into $\frac{n}{2}$ cylinders and 
the action of $T_{2n}$ is the composition of the square of the right Dehn twist along a core curve of the cylinder which contains the center of $\Pi_{2n}$ and the right Dehn twists along core curves of the other cylinders. 
If $n$ is odd,  $P_{2n}$ is decomposed into $\frac{n-1}{2}$ cylinders and 
the action of $T_{2n}$ is the composition of the right Dehn twists along core curves of all cylinders. 
Thus $\Gamma=\left<[R_{2n}],[T_{2n}]\right>$ is a subgroup of the Veech group $\Gamma(P_{2n}, \varphi_{2n})$. 
It is easy to see that $\Gamma$ is a $(n, \infty, \infty)$ triangle group.
Since only discrete group that contains $\Gamma$ is a $(2, 2n, \infty)$ triangle group (see \cite{EarGar97} and \cite{Singerman72})
and this cannot be $\Gamma(P_{2n}, \varphi_{2n})$, we have $\Gamma(P_{2n}, \varphi_{2n})=\left<[R_{2n}],[T_{2n}]\right>$.
\begin{figure}[h]
 \begin{center}
 \includegraphics*[keepaspectratio, scale=0.45]{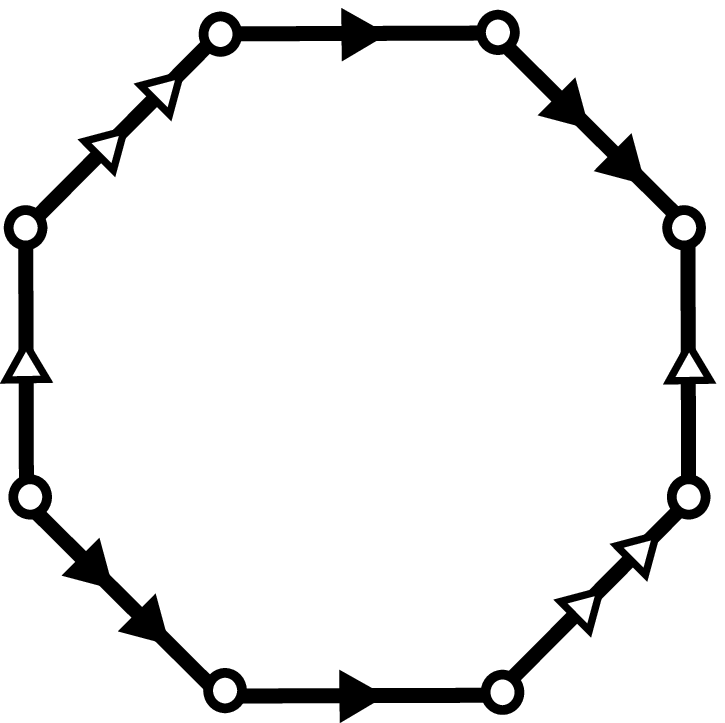}
\caption{}
\label{3}
 \end{center}
 \end{figure}
\end{example}

\section{Veech groups of coverings of $P_{2n}$ and Universal Veech group of $P_{2n}$}
Fix $n\geq4$. 
Let $P_{2n}$ be the same Riemann surface as in Example \ref{exampleofveech2} and $p : X\rightarrow P_{2n}$ be an unramified finite covering mapping.
Set $\varphi_X=p^*\varphi_{2n}$, here $\varphi_{2n}$ is the holomorphic quadratic differential on $P_{2n}$ defined in Example \ref{exampleofveech2}.
Our purpose is to calculate the Veech group $\Gamma(X, \varphi_X)$. 
We denote $\Gamma(X, \varphi_X)$ by $\Gamma(X)$ hereafter. 
Schmith\"usen constructed an algorithm for calculating Veech groups of origamis (\cite{Schmithusen04}).  
We apply her method to our case.

 Let $p_{2n} : {\widetilde X_{2n}} \rightarrow P_{2n}$ be the universal covering mapping and set $\widetilde \varphi_{2n}= p_{2n}^*\varphi_{2n}$.
 Note that $||\widetilde \varphi_{2n}||= +\infty$.
However, we can define the flat structure $u_{\widetilde \varphi_{2n}}$ on $\widetilde X_{2n}$ and the affine group 
$Aff^+(\widetilde X_{2n}, \widetilde \varphi_{2n})$ in the same manner as the case of integrable holomorphic quadratic differentials.
Moreover, we have a homomorphism $D : Aff^+(\widetilde X_{2n}, \widetilde \varphi_{2n}) \rightarrow {\rm PGL}(2, \mathbb{R})$.
Set $\Gamma(\widetilde X_{2n}) = {\rm Im}(D) \cap {\rm PSL}(2,\mathbb{R})$.
\begin{definition}[Universal Veech group of $P_{2n}$]
We call $\Gamma(\widetilde X_{2n})$ the universal Veech group of $P_{2n}$.
\end{definition}
\begin{remark}
Let $X$ be an unramified finite covering of $P_{2n}$. Then for each $f \in Aff^+(X, \varphi_X)$, there exists a lift $\widetilde f \in  Aff^+(\widetilde X_{2n}, \widetilde \varphi_{2n})$ 
with $D(\widetilde f)=D(f)$. Hence $\Gamma(X)$ is a subgroup of  $\Gamma(\widetilde X_{2n})$.
\end{remark}
 The following idea is due to Schmith\"usen (\cite{Schmithusen04}).
For each finite covering $X$ of $P_{2n}$, we take $\Gamma(X)$ as follows.\\
$\Gamma(X)=${\Big \{}$[A]\in \Gamma(\widetilde X_{2n})\ |\ \exists \widetilde f \in Aff^+(\widetilde X_{2n},\widetilde \varphi_{2n})\ {\rm s.t.} \ D(\widetilde f)=[A], \ \widetilde f$ is a lift of a homeomorphism of $X$ onto itself{\Big \}}\\
$=${\Big \{} $[A]\in \Gamma(\widetilde X_{2n})\ |\ \exists \widetilde f\in Aff^+(\widetilde X_{2n},\widetilde \varphi_{2n})\ {\rm s.t.} \ D(\widetilde f)=[A], \widetilde f_*({\rm Gal}(\widetilde X_{2n}/ X))={\rm Gal}(\widetilde X_{2n}/ X)$ {\Big \}}.

To understand $\Gamma(X)$, we determine $\Gamma(\widetilde X_{2n})$.
The following theorem is a main theorem of this paper.
\begin{theorem}\label{main1}
For all $n\geq 4$, $\Gamma(\widetilde X_{2n})=\left< [R_{2n}],[T_{2n}]\right>=\Gamma(P_{2n})$.
\end{theorem}

For the proof of theorem, we represent $A \in {\rm SL}(2,\mathbb{R})$ by 
\begin{center}
$A=\left(
  \begin{array}{cccc}
    r\cos\alpha(A)  & s\cos\beta (A) \\
    r\sin\alpha(A)  & s\sin\beta (A)
  \end{array}
  \right)$
  \end{center}
for some $r,s >0$ and $\alpha(A) , \beta(A) $ with $0\leq \alpha(A)<\beta (A)< 2\pi$.
And set $\theta(A)=\beta(A)-\alpha(A)$. 
This $\theta(A)$ means the angle of 
$A{  \left(
  \begin{array}{cccc}
    1 \\
    0
  \end{array}
  \right)}$ and 
$A{  \left(
  \begin{array}{cccc}
    0 \\
    1
  \end{array}
  \right)}$.

The following two lemmas give the proof of the theorem.
\begin{lemma}\label{Lemma1}
For $[A]\in\Gamma(\widetilde  X_{2n})$ with $|\cot\theta (A)|>\cot\frac{\pi}{2n}$, 
there exists $k,\ l \in \mathbb{Z}$ such that $|\cot\theta (A)|> |\cot\theta (T_{2n}^{l}R_{2n}^{k}A)|$.
\end{lemma}

\begin{lemma}\label{Lemma2}
For $[A]\in\Gamma(\widetilde  X_{2n})$ with $|\cot\theta (A)|>\cot\frac{\pi}{2n}$, 
there exists $B \in \left<R_{2n},T_{2n}\right>$ such that $\cot\frac{\pi}{2n}\geq |\cot\theta (BA)|$.
\end{lemma}

\begin{proof}[Proof of theorem \ref{main1}]
$\Gamma(P_{2n})\subseteq \Gamma(\widetilde X_{2n})$ is clear since $\Gamma(\widetilde X_{2n})$ is the universal Veech group of $P_{2n}$.
We show $\Gamma(\widetilde X_{2n})\subseteq \Gamma(P_{2n}) $.
By Lemma \ref{Lemma2}, for each $[A] \in \Gamma(\widetilde X_{2n})$, there exists 
$[B] \in \left< [R_{2n}],[T_{2n}]\right>$ such that
\begin{center}
$\cot\frac{\pi}{2n}\geq |\cot\theta (BA)|$.
\end{center}
If we map $Q_1$ of Figure \ref{4} by an affine transformation $BA$, the image is parallelogram whose vertices correspond to vertices of $2n$-gons and which has no such points in its interior.
Moreover,  it has the same area as $Q_1$ and each angle $\theta $ of its vertices satisfies $\pi/2n \leq \theta \leq \pi-\pi/2n$.
We can see that such parallelograms are only $Q_1, Q_2, Q_3$ and $Q_4$ of Figure \ref{4} up to the image of them by $[R_{2n}]$ and $[T_{2n}]$.
Then $BA$ is either
\begin{center}
$\left(
  \begin{array}{cccc}
   1 & 0\\
   0 & 1
  \end{array}
  \right)$, 
  $\left(
  \begin{array}{cccc}
   1 & \cot\frac{\pi}{2n}\\
   0 & 1
  \end{array}
  \right)$,
    $\left(
  \begin{array}{cccc}
   0 & \cot\frac{\pi}{2n}\\
   -\tan\frac{\pi}{2n} & 0
  \end{array}
  \right)$ or 
    $\left(
  \begin{array}{cccc}
   1 & \cot\frac{\pi}{2n}\\   
-\tan\frac{\pi}{2n} & 0
  \end{array}
  \right)$.
\end{center}
However, it does not happen except for the case that $AB$ is the identity 
 $I={ 
\left(
  \begin{array}{cccc}
    1 & 0 \\
    0 & 1 
  \end{array}
  \right)}$  
since every vertex of $2n$-gons must be mapped to a vertex.
Hence $BA=I$ and so $[A]=[B^{-1}] \in \left< [R_{2n}],[T_{2n}]\right>=\Gamma(P_{2n})$.
\end{proof}
\begin{figure}[h]
 \begin{center}
 \includegraphics*[keepaspectratio, scale=0.45]{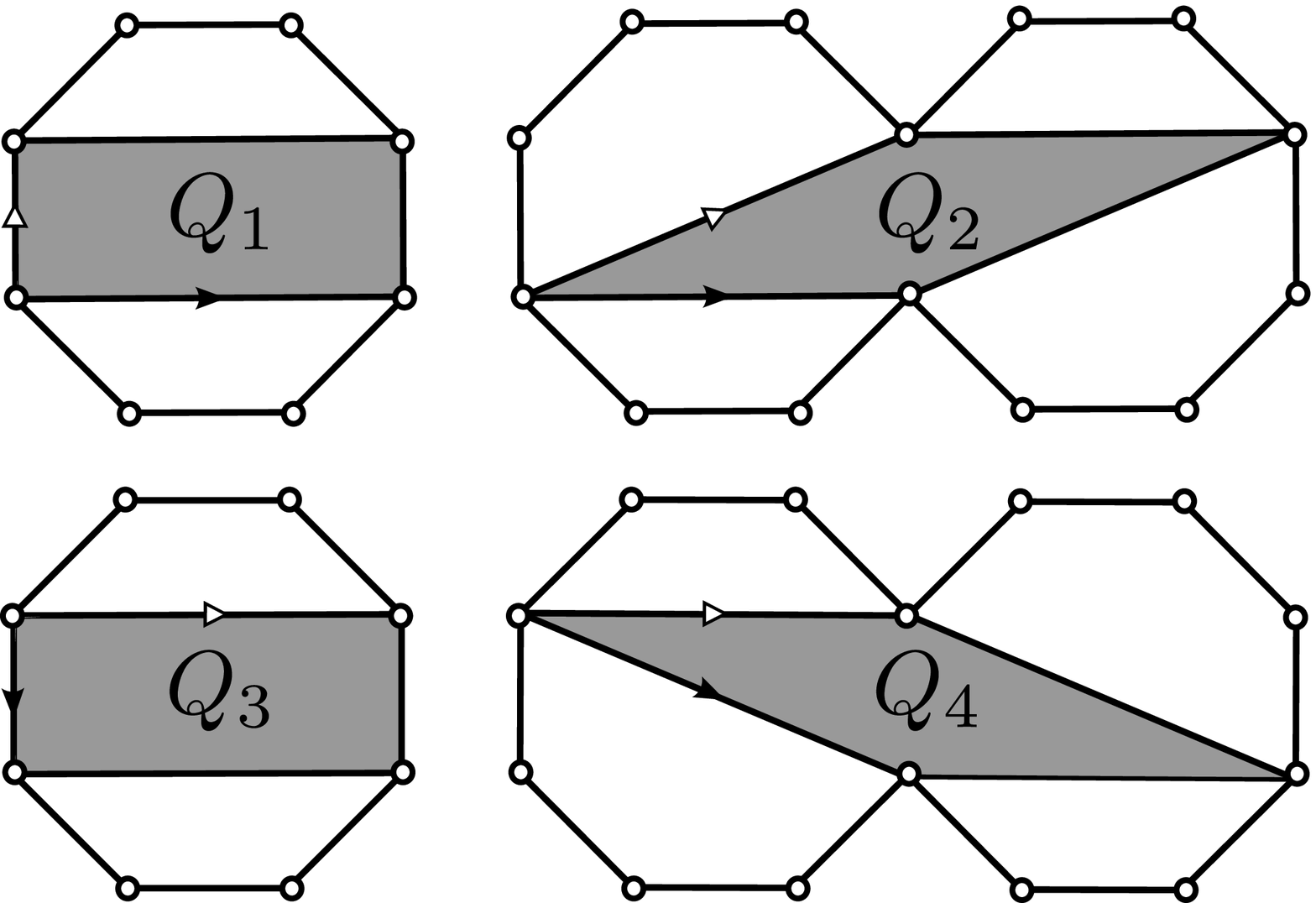}
\caption{}
\label{4}
 \end{center}
 \end{figure}

\begin{proof}[Proof of Lemma \ref{Lemma1}]
We consider two cases : (a) $\cot\theta (A)>\cot\frac{\pi}{2n}$ and (b) $-\cot\theta (A)>\cot\frac{\pi}{2n}$

{\bf Case (a)} : There exists $k\in \mathbb{Z}$ such that $B=R_{2n}^{k}A$ satisfies either $0\leq  \alpha(B)<\frac{\pi}{2n}$ or 
$\pi-\frac{\pi}{2n}\leq \alpha(B)<\pi$. 
We define the function 
\begin{eqnarray}
\nonumber
&F_{\alpha(B)}^{\beta (B)}(x)=&4\cot^2\frac{\pi}{2n}\cdot  \frac{\sin\beta (B)\sin\alpha (B)}{\sin(\beta (B)-\alpha (B))}\cdot x^2\\
\nonumber
&&+2\cot\frac{\pi}{2n}\cdot \frac{\sin(\beta (B)+\alpha (B))}{\sin(\beta (B)-\alpha (B))}\cdot x+\cot\theta (B)
\end{eqnarray}
of $x\in\mathbb{R}$. 
Note that $F_{\alpha(B)}^{\beta (B)}(l)=\cot\left(\beta (T_{2n}^{l}B)-\alpha (T_{2n}^{l}B)\right)=\cot\theta(T_{2n}^{l}B)$ for each $l \in \mathbb{Z}$. 

(a)-1 : If $0\leq  \alpha(B)<\frac{\pi}{2n}$, there exists $\widetilde f \in Aff^+(\widetilde X_{2n}, \widetilde \varphi_{2n})$ with $D(\widetilde f)=[B]$.
And  $\widetilde f$ maps the rectangle $Q_1$ of Figure \ref{4} to a parallelogram whose vertices correspond to vertices of $2n$-gons and which has no such points in its interior. Hence we have $0\leq  \alpha(B)<\beta (B) \leq \frac{\pi}{2n}$. 
From this, if $\alpha(B)=0$, then
\begin{equation}
\nonumber
F_0^{\beta (B)}(x)=2\cot\frac{\pi}{2n}\cdot x+\cot\theta (B)
\end{equation}
and
\begin{equation}
\nonumber
F_0^{\beta (B)}(-\frac{1}{2})=-\cot\frac{\pi}{2n}+\cot\theta (B)>0.
\end{equation}
So there exists a negative integer $l$ such that 
\begin{center}
$|F_{0}^{\beta (B)}(l)|< |F_{0}^{\beta (B)}(m)|$ \ for all $m \in \{0, -1,\cdot \cdot \cdot , l+1,l-1\}$.
\end{center}
 Now we have 
 \begin{center}
$|\cot\theta (A)|=|F_{0}^{\beta (B)}(0)|>|F_{0}^{\beta (B)}(l)|=|\cot\theta (T_{2n}^{l}R_{2n}^{k}A)|$.
\end{center}
If $0<\alpha(B)<\beta (B) \leq \frac{\pi}{2n}$, then $F_{\alpha(B)}^{\beta (B)}(x)$ is a  quadratic function of $x$ and the axis of $F_{\alpha(B)}^{\beta (B)}$ is 
\begin{equation}
\nonumber
x=-\frac{\cot\alpha (B)+\cot\beta  (B)}{4\cot\frac{\pi}{2n}}<-\frac{1}{2}
\end{equation}
and 
\begin{equation}
\nonumber
F_{\alpha(B)}^{\beta (B)}(-\frac{1}{2})>0.
\end{equation}
Hence there exists a negative integer $l$ such that 
\begin{center}
$|F_{\alpha(B)}^{\beta (B)}(l)|< |F_{\alpha(B)}^{\beta (B)}(m)|$ \ for all $m \in \{0, -1,\cdot \cdot \cdot , l+1,l-1\}$.
\end{center}
And we have $|\cot\theta (A)|=|F_{\alpha(B)}^{\beta (B)}(0)|> |F_{\alpha(B)}^{\beta (B)}(l)|=|\cot\theta (T_{2n}^{l}R_{2n}^{k}A)|$.\\

(a)-2 : If $\pi-\frac{\pi}{2n}\leq \alpha (B)<\pi$,  we have $\pi-\frac{\pi}{2n}\leq \alpha (B)< \beta(B) \leq \pi$
and  
\begin{equation}
\nonumber
F_{\alpha(B)}^{\beta (B)}(x)=F_{(\pi-\beta(B))}^{(\pi-\alpha (B))}(-x).
\end{equation}
By using the argument of (a)-1, we have $|\cot\theta (A)|>|\cot\theta (T_{2n}^{l}R_{2n}^{k}A)|$ for some $l\in \mathbb{Z}$.

{\bf Case (b)} :  We apply the same argument as in the Case (a) to the angle of two vectors $A{  \left(
  \begin{array}{cccc}
    1 \\
    0
  \end{array}
  \right)}$ and 
$A{  \left(
  \begin{array}{cccc}
    0 \\
    -1
  \end{array}
  \right)}$.
Then we have $|\cot\theta (A)|>|\cot\theta (T_{2n}^{l}R_{2n}^{k}A)|$ for some $k,l\in \mathbb{Z}$. 
\end{proof}

\begin{proof}[Proof of Lemma \ref{Lemma2}]
Let $[A]$ be an element in $\Gamma(\widetilde X_{2n})$ with $|\cot\theta (A)|>\cot\frac{\pi}{2n}$.
From the proof of Lemma \ref{Lemma1}, we obtain $A_1=T_{2n}^{l_1}R_{2n}^{k_1} A$ with $|\cot\theta (A_1)|<|\cot\theta (A)|$ for some $k_1\in \mathbb{Z}$ and $l_1 \in \mathbb{Z}-\{0\}$. 
If $|\cot\theta (A_1)|>\cot\frac{\pi}{2n}$, then we obtain  $A_2=T_{2n}^{l_2}R_{2n}^{k_2} A_1$ with $|\cot\theta (A_2)|<|\cot\theta (A_1)|$ for some $k_2,l_2 \in \mathbb{Z}-\{0\}$ from the proof of Lemma \ref{Lemma1}  again. 
We repeat this operation.
If there exists $m_0 \in  \mathbb{N}$ such that $\cot\frac{\pi}{2n}\geq |\cot\theta (A_{m_0})|$ holds, then $B=A_{m_0}A^{-1}$ is what we want.  
Suppose that $|\cot\theta (A_m)|>\cot\frac{\pi}{2n}$ holds for every $m \in \mathbb{N}$.
Then we have an infinite sequence $\{A_m\}$ in $\left<R_{2n},T_{2n}\right>\cdot A$ with  $|\cot\theta (A_{m-1})|>|\cot\theta (A_m)|>\cot\frac{\pi}{2n}$ for all $m$.
We represent $A_m$ by 
\begin{center}
$A_m=\left(
  \begin{array}{cccc}
    r_m\cos\alpha_m  & s_m\cos\beta_m \\
    r_m\sin\alpha_m  & s_m\sin\beta_m
  \end{array}
  \right)$
\end{center}
for some $r_m,s_m >0$ and $0\leq \alpha_m<\beta_m < 2\pi$ with $r_ms_m\sin(\beta_m-\alpha_m )=1$. 
For each $m$, there exists $\widetilde  f_m \in Aff^+(\widetilde X_{2n}, \widetilde \varphi_{2n})$ such that $D(\widetilde  f_m)=[A_m]$ and $\widetilde  f_m$ maps Euclidean segment which connect vertices of $2n$-gons to other segment. 
Thus we have 
\begin{center}
$r_m=\Big|
A_m  
\left(
  \begin{array}{cccc}
    1\\
    0
  \end{array}
  \right)
\Big| \geq 1$ 
and
$s_m=\Big|
A_m 
\left(
  \begin{array}{cccc}
    0\\
    1
  \end{array}
  \right)
\Big| \geq 1$.
  \end{center}
  Moreover, we have 
\begin{equation}
\nonumber
r_ms_m=\frac{1}{\sin(\beta_m-\alpha_m )}\leq \frac{1}{\sin(\beta_1-\alpha_1 )}.
\end{equation}
Hence $\{\alpha _m\}, \{\beta _m\}, \{r_m\}$ and $\{s_m\}$ are bounded and there exists a subsequence $\{A_{m_{i}}\}$ of $\{A_m\}$ such that $A_{m_{i}}$ converges to some $A_\infty \in {\rm SL}(2,\mathbb{R})$. 
Since $\{A_{m_i}\}$ is in a discrete set  $\left<R_{2n},T_{2n}\right> \cdot A$ , there exists $i_0 \in \mathbb{N}$ such that $A_{m_i}=A_\infty$ for all $i\geq i_0$.
However, this contradicts the construction of the sequence $\{A_m\}$. 
Hence there exists $m_0 \in \mathbb{N}$ such that $\cot\frac{\pi}{2n}\geq |\cot\theta (A_{m_0})|$.
\end{proof}

\section{Calculation of Veech groups}
Let $X$ be an unramified finite covering of $P_{2n}$.
By theorem \ref{main1}, we can write $\Gamma(X)$ as follows.
\begin{flushright}
$\Gamma(X)=$ {\Big \{} $[A]\in \left<[R_{2n}], [T_{2n}] \right>\ |\ \exists \widetilde f\in Aff^+(\widetilde X_{2n},\widetilde \varphi_{2n})\ s.t.\ D(\widetilde f)=[A], $ 
$\widetilde f_*({\rm Gal}(\widetilde X_{2n}/ X))={\rm Gal}(\widetilde X_{2n}/ X)$ {\Big \}}.
\end{flushright}

Let $z_0$ be the point of $P_{2n}$ which corresponds to the center of the $2n$-gon $\Pi_{2n}$ as in Example \ref{exampleofveech2} and $\overline z_0$ be one of the preimages of $z_0$ in $X$. 
Let $\{x_1, x_2, \cdots, x_n\}$ be the system of generators of $\pi_1(P_{2n}, z_0)$ as Figure \ref{5}.
Then $R_{2n}$ and $T_{2n}$ define the following  automorphisms $\gamma_{R_{2n}}$ and $\gamma_{T_{2n}}$ on $\pi_1(P_{2n}, z_0)$ (see Example \ref{exampleofveech2}).
\begin{eqnarray}
\gamma_{R_{2n}} :\left\{
\begin{array}{l}
\nonumber
x_i\mapsto x_{i+1}\ \ (i=1, 2, \cdots, n-1)\\
x_n\mapsto x_1^{-1}
\end{array}.
\right.
\end{eqnarray}
If $n$ is even, 
\begin{eqnarray}
\gamma_{T_{2n}} :\left\{
\begin{array}{l}
\nonumber
x_1\mapsto x_1 \\
x_{n+2-i}^{-1}x_{i}\mapsto x_{n+2-i}^{-1}x_{i}\ \ (i=2,3,\cdots , \frac{n}{2})\\
x_i\mapsto (x_{n+2-i}^{-1}x_{i})\cdots (x_{n-1}^{-1}x_{3})(x_{n}^{-1}x_{2})x_1^2 x_i \ \ (i=2,3,\cdots , \frac{n}{2})\\
x_{\frac{n}{2}+1}\mapsto (x_{\frac{n}{2}+2}^{-1}x_\frac{n}{2})\cdots (x_{n-1}^{-1}x_{3})(x_{n}^{-1}x_{2})x_1^2 x_{\frac{n}{2}+1}
\end{array}
\right.
\end{eqnarray}
and if $n$ is odd, 
\begin{eqnarray}
\gamma_{T_{2n}} :\left\{
\begin{array}{l}
\nonumber
x_{n+1-i}^{-1}x_{i}\mapsto x_{n+1-i}^{-1}x_{i}\ \ (i=1,2,\cdots , \frac{n-1}{2}).\\
x_i\mapsto (x_{n+1-i}^{-1}x_{i})\cdots (x_{n-1}^{-1}x_{2})(x_{n}^{-1}x_{1}) x_i \ \ (i=1,2,\cdots , \frac{n-1}{2})\\
x_\frac{n+1}{2}\mapsto (x_{\frac{n+3}{2}}^{-1}x_\frac{n-1}{2})\cdots (x_{n-1}^{-1}x_{2})(x_{n}^{-1}x_{1}) x_\frac{n+1}{2}.
\end{array}
\right.
\end{eqnarray}
Since ${\rm Gal}(\widetilde X_{2n}/ P_{2n}) < {\rm Ker}(D)$,  
${\rm Ker}(D)/{\rm Gal}(\widetilde X_{2n}/ P_{2n}) =\{ [id], [\widetilde h^{n}]\}$ for some $\widetilde h \in Aff^+(\widetilde X_{2n},\widetilde \varphi_{2n})$ with $D(\widetilde h)=[R_{2n}]$ and 
each element in ${\rm Gal}(\widetilde X_{2n}/ P_{2n})$ defines an inner automorphism of ${\rm Gal}(\widetilde X_{2n}/ P_{2n})\cong \pi_1(P_{2n}, z_0)$,  
the action of each element of $Aff^+(\widetilde X_{2n},\widetilde \varphi_{2n})$ on $\pi_1(P_{2n}, z_0)$ can be represented by a composition of $\gamma_{R_{2n}}$, $\gamma_{T_{2n}}$ and inner automorphisms of $\pi_1(P_{2n}, z_0)$.

\begin{figure}[h]
 \begin{center}
 \includegraphics*[keepaspectratio, scale=0.45]{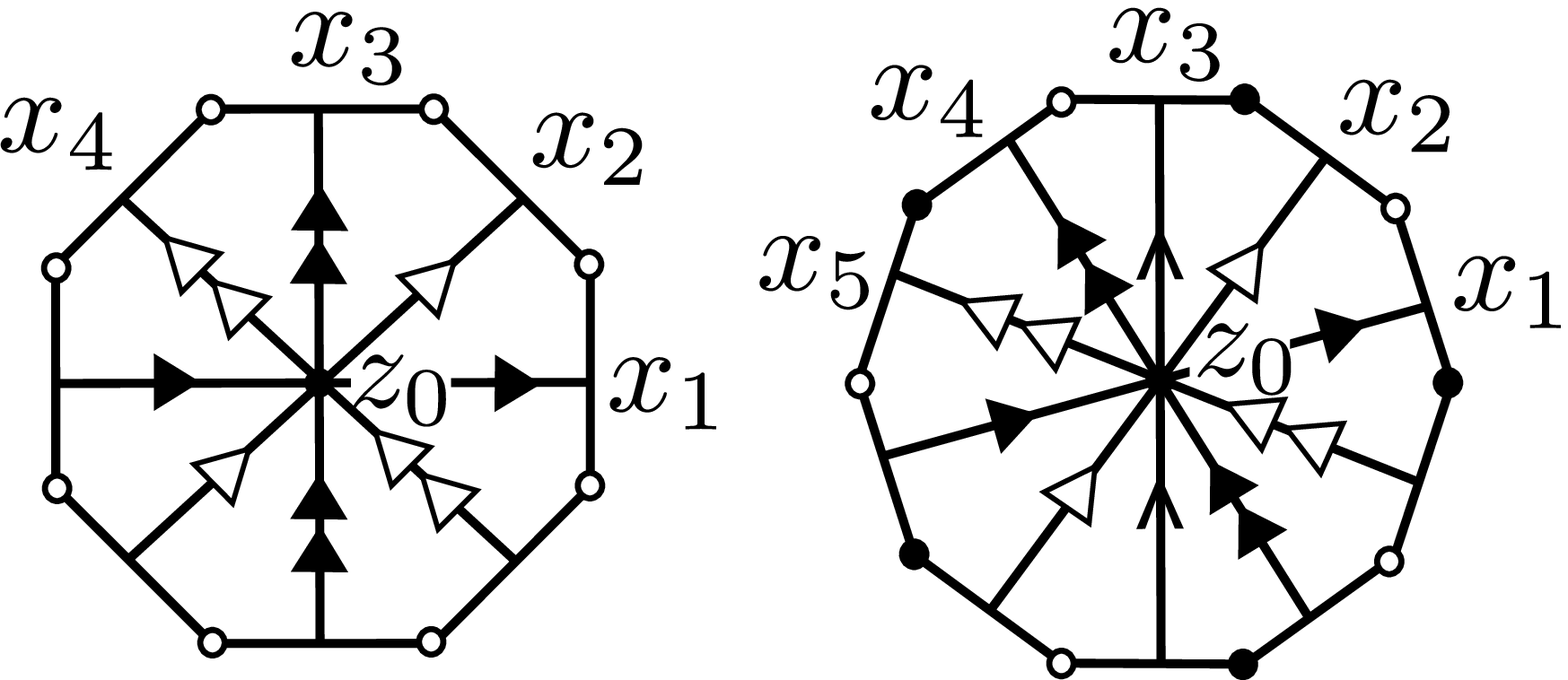}
\caption{}
\label{5}
 \end{center}
 \end{figure} 
 
Hence we have the following.
\begin{proposition}
For $\widetilde f \in Aff^+(\widetilde X_{2n}, \widetilde \varphi_{2n})$, following two are equivalent. 
Here $A$ is one of elements in $D(\widetilde f)$. 
\begin{itemize}
\item The mapping $\widetilde f$ satisfies $\widetilde f_*({\rm Gal}(\widetilde X_{2n}/ X))={\rm Gal}(\widetilde X_{2n}/ X)$.
\item There exists one of the preimages $\overline z_1 \in X$ of $z_0$ such that 
\begin{center}
 $\gamma_A(\pi_1(X, \overline z_0))=\pi_1(X, \overline z_1)$ or $\gamma_{-A}(\pi_1(X, \overline z_0))=\pi_1(X, \overline z_1)$. 
 \end{center}
\end{itemize}

\end{proposition}

By using this condition, we can determine whether $[A]$ is in $\Gamma(X)$ or not for each $[A] \in \left<[R_{2n}], [T_{2n}] \right>$.

Now we can calculate the Veech group $\Gamma(X)$ of an unramified finite covering $X$ of $P_{2n}$ by using the following method. 
Schmith\"usen(\cite{Schmithusen04}) also use this method to the calculations of Veech groups of origamis. 
The calculation is done on the following tree which we explain below.
\begin{figure}[h]
 \begin{center}
 \includegraphics*[keepaspectratio, scale=0.5]{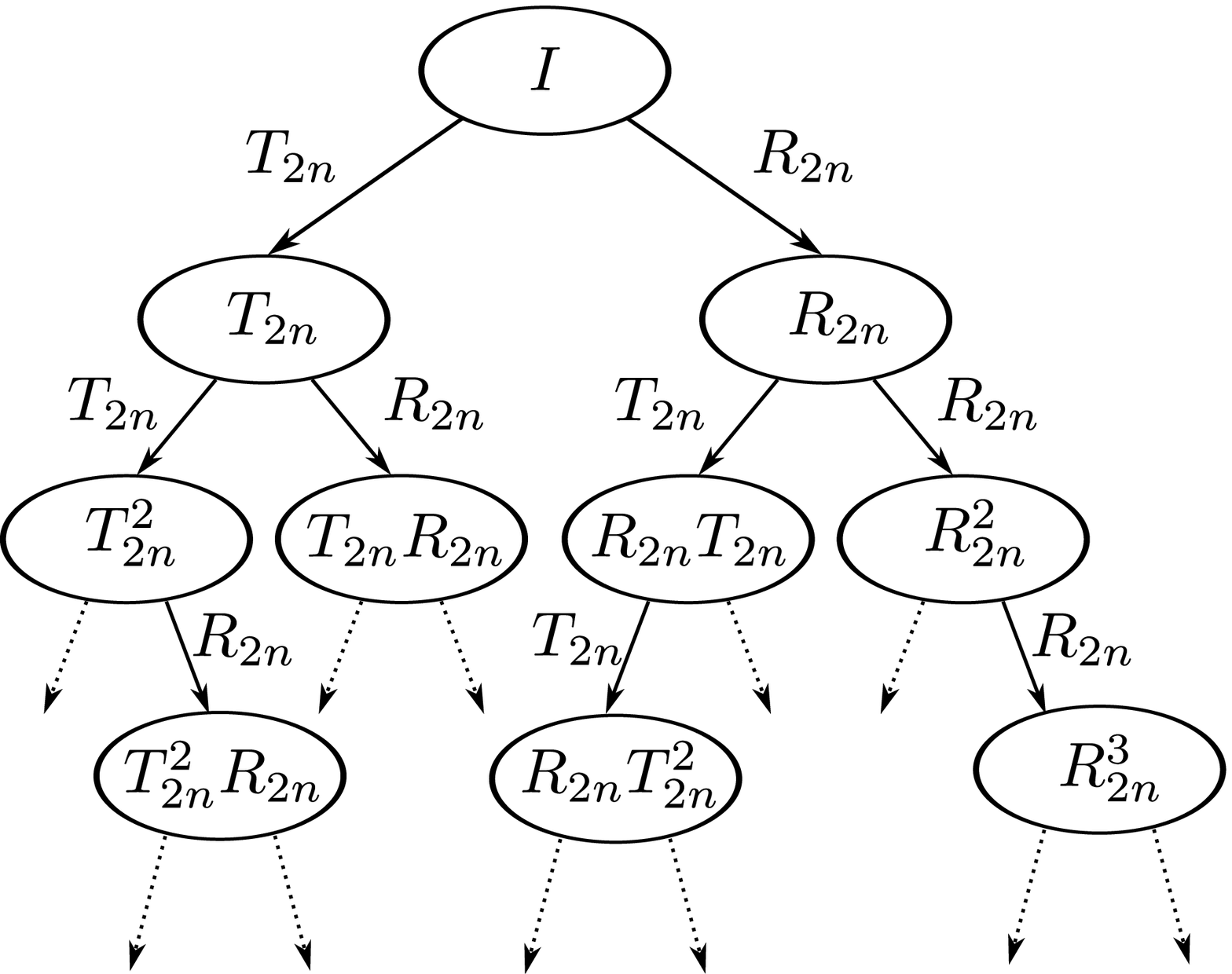}
 \label{Ract}
 \end{center}
\end{figure}

\ovalbox{
\begin{tabular}{l}
\underline {Calculation of $\Gamma(X)$}(Reidemeister-Schreier method).\\
\\
Given an unramified finite covering $X$ of $P_{2n}$.\\
Let {\bf Rep} and {\bf Gen} be empty sets.\\
Add $[I]$ to {\bf Rep}. Set $A=I$.\\
Loop:\\
Set $B=A\cdot T_{2n}$, $C=A\cdot R_{2n}$.\\
Check whether $B$ is already represented by {\bf Rep}:\\
For each $[D]$ in {\bf Rep}, check whether $[B]\cdot[ D]^{-1}$ is in $\Gamma(X)$.\\
If so, add $[B]\cdot [D]^{-1}$ to {\bf Gen}.\\
If none is found, add $[B]$ to {\bf Rep}.\\
Do the same for $C$ instead of $B$.\\
If there exists a successor of $A$ in {\bf Rep}, \\
let $A$ be this successor and go to the beginning of the loop.\\
If not, finish the loop.\\
\\
Result: \\
{\bf Gen} : a list of generators of $\Gamma(X)$.\\
{\bf Rep} : a list of coset representatives in $\left<[R_{2n}],[T_{2n}]\right>$.
\end{tabular}}

\begin{proposition} \label{R-S-method}
Let $X$ be an unramified finite covering of $P_{2n}$. Then we have the following properties.
\begin{enumerate} 
\item[(1)] Any two elements in {\bf Rep} belong to different cosets.
\item[(2)] The calculation stops in finitely many steps.
\item[(3)] In the end, each coset is represented by a member of {\bf Rep}.
\item[(4)] In the end, $\Gamma(X)$ is generated by the elements in {\bf Gen}.
\end{enumerate}  
\end{proposition}
\begin{proof}
(1) is clear and we can see a proof of (3), (4) in \cite{Schmithusen04}.
(2) is equivalent to what $\Gamma(X)$ is a finite index subgroup of $\left<[R_{2n}], [T_{2n}]\right>$.
By the next proposition , we conclude that $\Gamma(X)$ and $\Gamma(P_{2n})=\left<[R_{2n}], [T_{2n}]\right>$ are commensurable.
Hence $\Gamma(X)$ is a finite index subgroup of $\left<[R_{2n}],[T_{2n}]\right>$.
Since all elements in {\bf Rep} belong to different cosets of $\Gamma(X)$ in $\left<[R_{2n}], [T_{2n}]\right>$, 
$\sharp{\rm \bf Rep}$ cannot be greater than this index and hence the calculation of $\Gamma(X)$ stops in finitely many steps.
\end{proof}

For a Riemann surface $X$ and a holomorphic quadratic differential $\varphi$, denote by $C(X, \varphi)$ the set of all zeros of $\varphi$ and punctures of $X$. 
\begin{proposition}$($\cite{GutJud96} and \cite{GutJud00}.$)$
Let $p: X \rightarrow Y$ be a covering mapping between  Riemann surfeces.
Let $\varphi_X$ be a holomorphic quadratic differential on $X$ and set $\varphi_Y=p_*\varphi_X$.
Suppose that $p(C(Y,\varphi_Y))=C(X,\varphi_X)$ and $p^{-1}(C(X,\varphi_X))=C(Y,\varphi_Y)$.
Then the Veech groups $\Gamma(X, \varphi_X)$ and $\Gamma(Y, \varphi_Y)$ are commensurable.
\end{proposition}

\begin{example} Let $X$ be the covering of $P_8$ as Figure \ref{6}.
We calculate the Veech group $\Gamma(X)$.
\begin{figure}[h]
 \begin{center}
 \includegraphics*[keepaspectratio, scale=0.439]{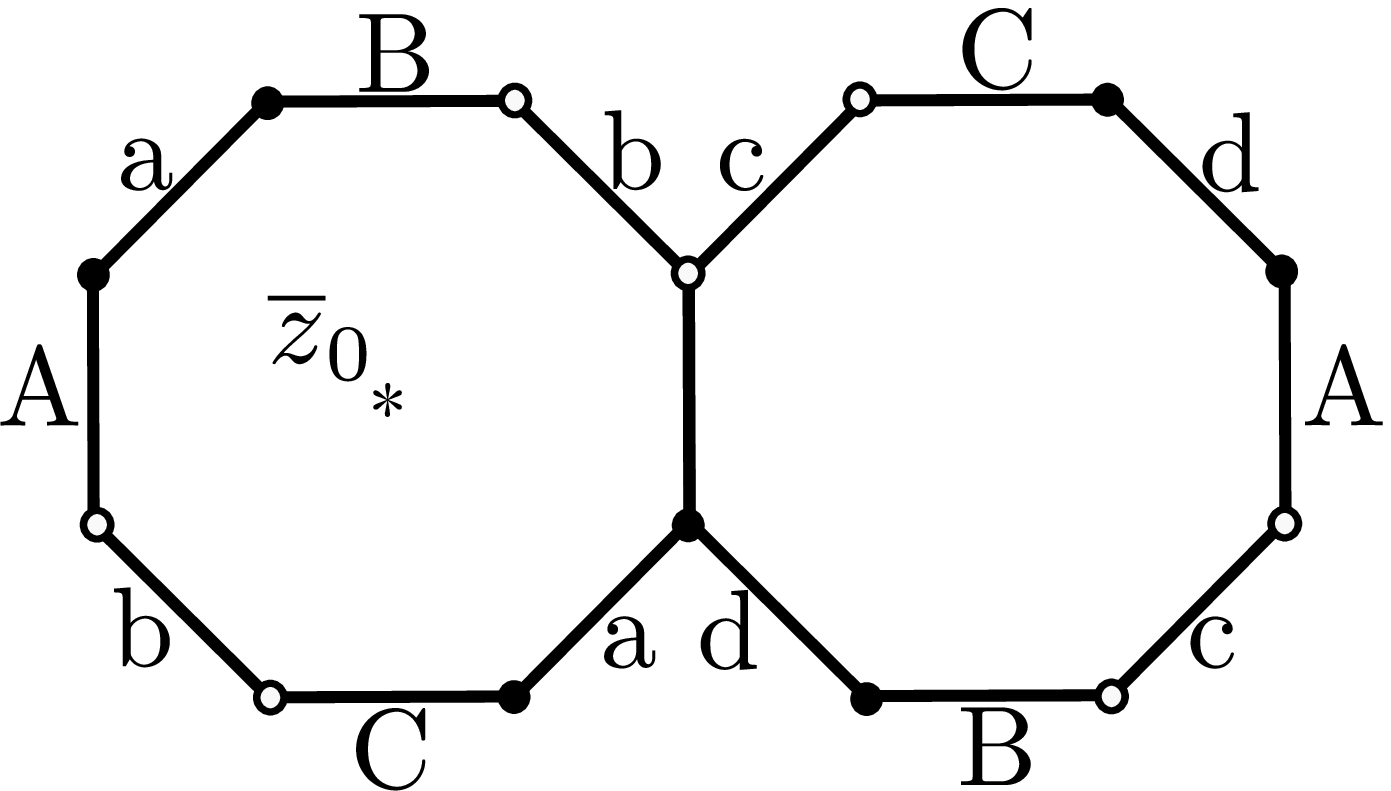}
\caption{}
\label{6}
 \end{center}
\end{figure}\\
The fundamental group of $X$ is 
\begin{equation}
\nonumber
\pi_1(X,\overline z_0)=\left< x_1^2\ ,\ x_2\ ,\ x_4\ ,\ x_1x_3\ ,\ x_3x_1\ ,\ x_1^{-1}x_2x_1\ ,\ x_1^{-1}x_4x_1 \right>.
\end{equation}
\begin{enumerate} 
\item[Loop 1 :] {\bf Rep}=$\{[I] \}$, {\bf Gen}=$\phi $, $A=I$, $B=T_8$, $C=R_8$. \\
We check $[B]\cdot [I]^{-1}=[T_8]$ $;$ 
the homomorphism $\gamma_{T_8}$ maps the generators of $\pi_1(X,\overline z_0)$ as follows
\begin{eqnarray}
\gamma_{T_8} :\left\{
\begin{array}{ll}
\nonumber 
x_1^2\mapsto x_1^2\\
x_2\mapsto x_4^{-1}x_2 x_1^2 x_2\\
x_4\mapsto x_4^{-1}x_2 x_1^2 x_4 & \\
x_1 x_3\mapsto x_1 x_4^{-1}x_2 x_1^2 x_3\\
x_3 x_1\mapsto x_4^{-1}x_2 x_1^2 x_3 x_1\\
x_1^{-1}x_2 x_1\mapsto x_1^{-1}x_4^{-1}x_2 x_1^2 x_2 x_1\\
x_1^{-1}x_4 x_1\mapsto x_1^{-1}x_4^{-1}x_2 x_1^2 x_4 x_1
\end{array}.
\right.
\end{eqnarray}
By taking $\overline z_0$ as a base point, all images represent closed curves. 
Hence $[T_8]$ is an element in $\Gamma(X)$ and add $[T_8]$ in {\bf Gen}.\\

Now {\bf Rep}=$\{[I] \}$, {\bf Gen}=$\{[T_8]\}$. \\
We check $[C]\cdot [I]^{-1}=[R_8]$ $;$ 
there is no point of $X$ such that $\gamma_{R_{2n}}(x_2)=x_3$ or $\gamma_{-R_{2n}}(x_2)=x_3^{-1}$ 
represent closed curves with the point as a base point. 
Hence $[R_8]$ is not in $\Gamma(X)$. We add $[R_8]$ in {\bf Rep}.\\

Now {\bf Rep}=$\{[I], [R_8] \}$, {\bf Gen}=$\{[T_8]\}$ and $R_8$ is a successor of $A=I$ and is in {\bf Rep}. We set $A=R_8$. 
\item[Loop 2 :] {\bf Rep}=$\{[I], [R_8] \}$, {\bf Gen}=$\{[T_8]\}$, $A=R_8$, $B=R_8T_8$, $C=R_8^2$. \\
We check $[B]\cdot [I]^{-1}=[R_8T_8]$ $;$ it is not in $\Gamma(X)$.\\
We check $[B]\cdot [R_8]^{-1}=[R_8T_8R_8^{-1}]$ $;$ the homomorphism $\gamma_{R_8T_8R_8^{-1}}$ is the form 
\begin{eqnarray}
\gamma_{R_8 T_8 R_8^{-1}} :\left\{
\begin{array}{l}
\nonumber
x_1\mapsto x_1 x_2^{-2} x_3^{-1} x_1^{-1}\\
x_2\mapsto x_2\\
x_3\mapsto x_1 x_3 x_2^2 x_3\\
x_4\mapsto x_1 x_3 x_2^2 x_4
\end{array}
\right.
\end{eqnarray}
and maps the generators of $\pi_1(X,\overline z_0)$ as follows
\begin{eqnarray}
\gamma_{R_8 T_8 R_8^{-1}} :\left\{
\begin{array}{ll}
\nonumber 
x_1^2\mapsto  x_1 x_2^{-2} x_3^{-1} x_2^{-2} x_3^{-1} x_1^{-1}\\
x_2\mapsto x_2 \\
x_4\mapsto x_1 x_3 x_2^2 x_4\\
x_1 x_3\mapsto x_1 x_3\\
x_3 x_1\mapsto x_1 x_3 x_2^2 x_3 x_1 x_2^{-2} x_3^{-1} x_1^{-1}\\
x_1^{-1} x_2 x_1\mapsto x_1 x_3 x_2^2 x_1^{-1} x_2 x_1 x_2^{-2} x_3^{-1} x_1^{-1}\\
x_1^{-1} x_4 x_1\mapsto x_1 x_3 x_2^2 x_3 x_2^2 x_4 x_1 x_2^{-2} x_3^{-1} x_1^{-1}
 \end{array}.
\right.
\end{eqnarray}
By taking $\overline z_0$ as a base point, all images represent closed curves. 
Hence $[R_8T_8R_8^{-1}]$ is an element in $\Gamma(X)$ and add $[R_8T_8R_8^{-1}]$ in {\bf Gen}.\\

Now {\bf Rep}=$\{[I], [R_8] \}$, {\bf Gen}=$\{[T_8], [R_8T_8R_8^{-1}]\}$. \\
We check $[C]\cdot [I]^{-1}=[R_8^2]$ $;$
the homomorphism $\gamma_{R_8}$ maps the generators of $\pi_1(X,\overline z_0)$ as follows
\begin{eqnarray}
\gamma_{R_8^2} :\left\{
\begin{array}{ll}
\nonumber 
x_1^2\mapsto x_3^2 \\ 
x_2\mapsto x_4 & \\
x_4\mapsto x_2^{-1} \\
x_1 x_3\mapsto x_3 x_1^{-1}\\
x_3 x_1\mapsto x_1^{-1} x_3\\
x_1^{-1} x_2 x_1\mapsto x_3^{-1} x_4 x_3\\
x_1^{-1} x_4 x_1\mapsto x_3^{-1} x_2^{-1} x_3
\end{array}.
\right.
\end{eqnarray}
By taking $\overline z_0$ as a base point, all images represent closed curves. 
Hence $[R_8^2]$ is an element in $\Gamma(X)$ and add $[R_8^2]$ in {\bf Gen}.\\
We check $[C]\cdot [R_8]^{-1}=[R_8]$ $;$ it is not in $\Gamma(X)$.\\

Now,  {\bf Rep}=$\{[I], [R_8] \}$, {\bf Gen}=$\{[T_8], [R_8T_8R_8^{-1}], [R_8^2]\}$ and there is no successor of $A=R_8$ in {\bf Rep}. We finish the loop.
\end{enumerate} 

Result $:$ {\bf Rep}=$\{[I], [R_8] \}$, {\bf Gen}=$\{[T_8], [R_8T_8R_8^{-1}], [R_8^2]\}$.\\
As a result, $\Gamma(X)=\left< [T_8], [R_8T_8R_8^{-1}], [R_8^2] \right>$ and coset representatives in  $\left<[R_8],[T_8]\right>$ is $\{[I], [R_8] \}$.\\
\end{example}

\begin{remark}

In the case of origamis, Schmith\"usen showed that the calculations always stop 
by connecting the Veech groups of origami with subgroups of  ${\rm SL}(2, \mathbb{Z})$(see \cite{Schmithusen04}).
In our case, for certain Abelian coverings of $2n$-gons, we connect the Veech groups with subgroups of ${\rm SL}(n,\mathbb{Z}_d)$ and calculate the Veech groups by using the corresponding matrices.
It is seen in section \ref{Abel}.
\end{remark}

\section{Calculation of $\mathbb{H}/\Gamma(X)$}
Let $X$ be an unramified finite covering of $P_{2n}$. 
Assume that the calculation of $\Gamma(X)$ by the Reidemeister-Schreier method stopped. 
Then  {\bf Gen} is a list of generators of $\Gamma(X)$ and {\bf Rep} is a list of coset representatives in $\left<[R_{2n}],[T_{2n}]\right>$.

Let $D$ be the fundamental domain of $\left<[R_{2n}],[T_{2n}]\right>$ in $\mathbb{H}$ as Figure \ref{7}. 
Then \begin{center}
$F \displaystyle={\rm Int} \Bigl(\bigcup_{[A]\in {\bf Rep}}[A](\overline D)\Bigr)$
\end{center}
is a fundamental domain of $\Gamma(X)$. 
Here $[A]$ means a M\"obius transformation.

\begin{figure}[h]
 \begin{center}
 \includegraphics*[keepaspectratio, scale=0.5]{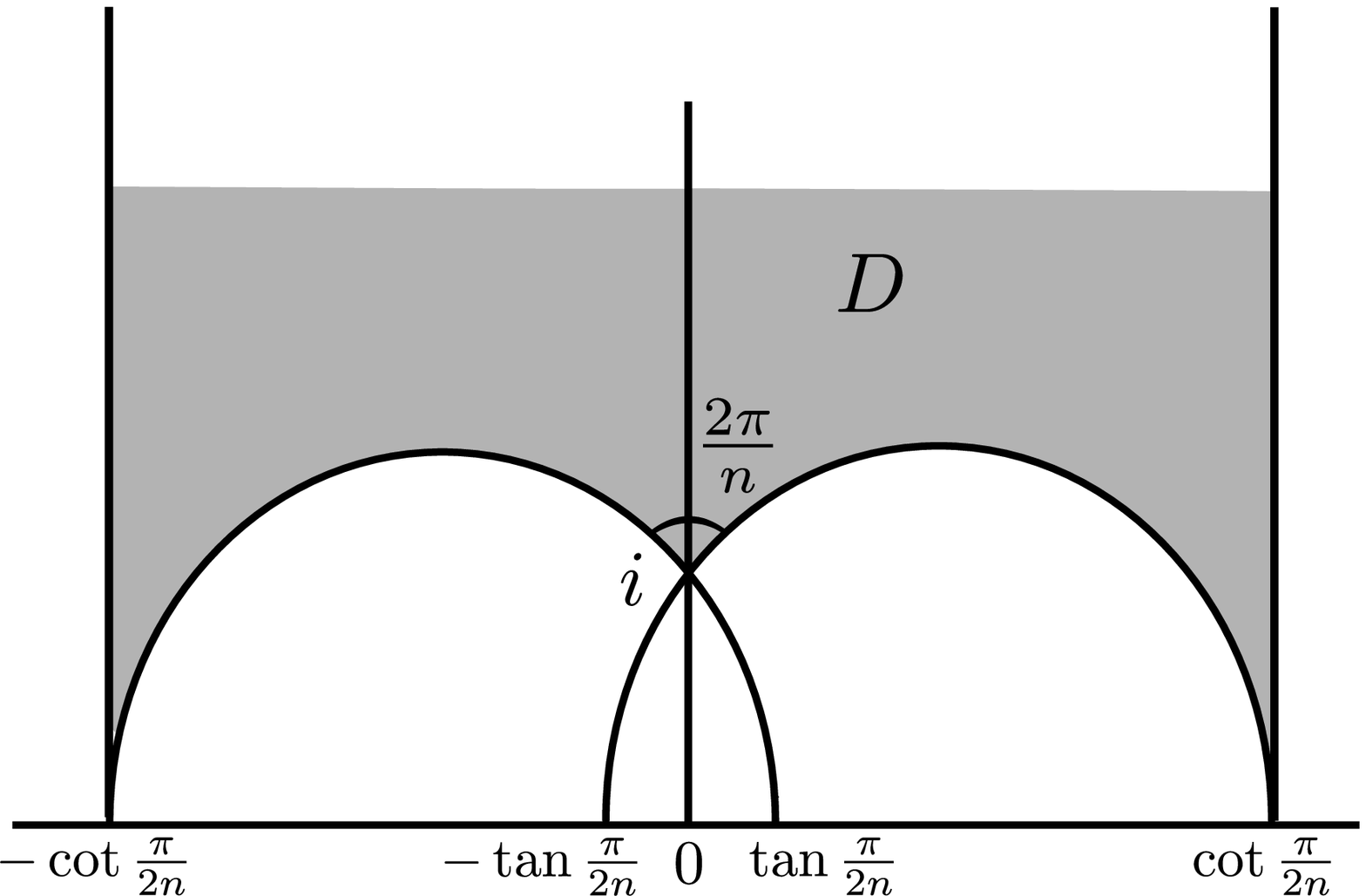}
\caption{}
\label{7}
 \end{center}
\end{figure}
By reading {\bf Gen},  we can know all pairs of sides of $F$ which are identified by the action of $\Gamma(X)$. 
We call sides of $[A](D)$ which correspond to $(-\cot \frac{\pi}{2n},i)$, $(\cot\frac{\pi}{2n},i)$, $(-\cot\frac{\pi}{2n},i\infty)$ and $(\cot\frac{\pi}{2n},i\infty)$
the $R^{-1}$-side, the  $R$-side, the $T^{-1}$-side and the $T$-side of $[A]$, respectively.
\begin{proposition}\label{side}
Assume that ${\rm \bf Rep}=\{[A_1], [A_2], \cdot \cdot \cdot , [A_k]\}$. 
Then for each $i, j \in \{1,2, \cdot \cdot \cdot k\}$,
\begin{itemize}
\item The $T$-side of $[A_j]$ and the $T^{-1}$-side of $[A_i]$ are identified if and only if $[A_jT_{2n}A_i^{-1}] \in \Gamma(X)$.
\item The $R$-side of $[A_j]$ and the $R^{-1}$-side of $[A_i]$ are identified if and only if $[A_jR_{2n}A_i^{-1}] \in \Gamma(X)$.
\end{itemize}
\end{proposition}
We give a triangulation of $\mathbb{H}/\Gamma(X)$ by decomposing $D$ as Figure \ref{8}.
Then the number of triangles and sides are $2\cdot \sharp{\rm \bf Rep}$ and $3\cdot \sharp{\rm \bf Rep}$, respectively.
Moreover, we can calculate the number $v$ of vertices by using Proposition \ref{side}.
When we calculate $v$, we decompose $v$ as $v=v_\infty+v_{\cot}+v_{\rm cone}$. 
Here $v_\infty$ is the number of vertices corresponding to $\infty$ of $D$, 
$v_{\cot}$ is the number of vertices corresponding to $\pm \cot \frac{\pi}{2n}$ of $D$ and
$v_{\rm cone}$ is the number of vertices corresponding to $i$ of $D$.
Then $\mathbb{H}/\Gamma(X)$ has genus $(2+ \sharp{\rm \bf Rep}-v)/2$ and $v_\infty+v_{\cot}$ punctures.
We can also calculate the number of cone points and their orders in the calculation of $v_{\rm cone}$. 
\begin{figure}[h]
 \begin{center}
 \includegraphics*[keepaspectratio, scale=0.5]{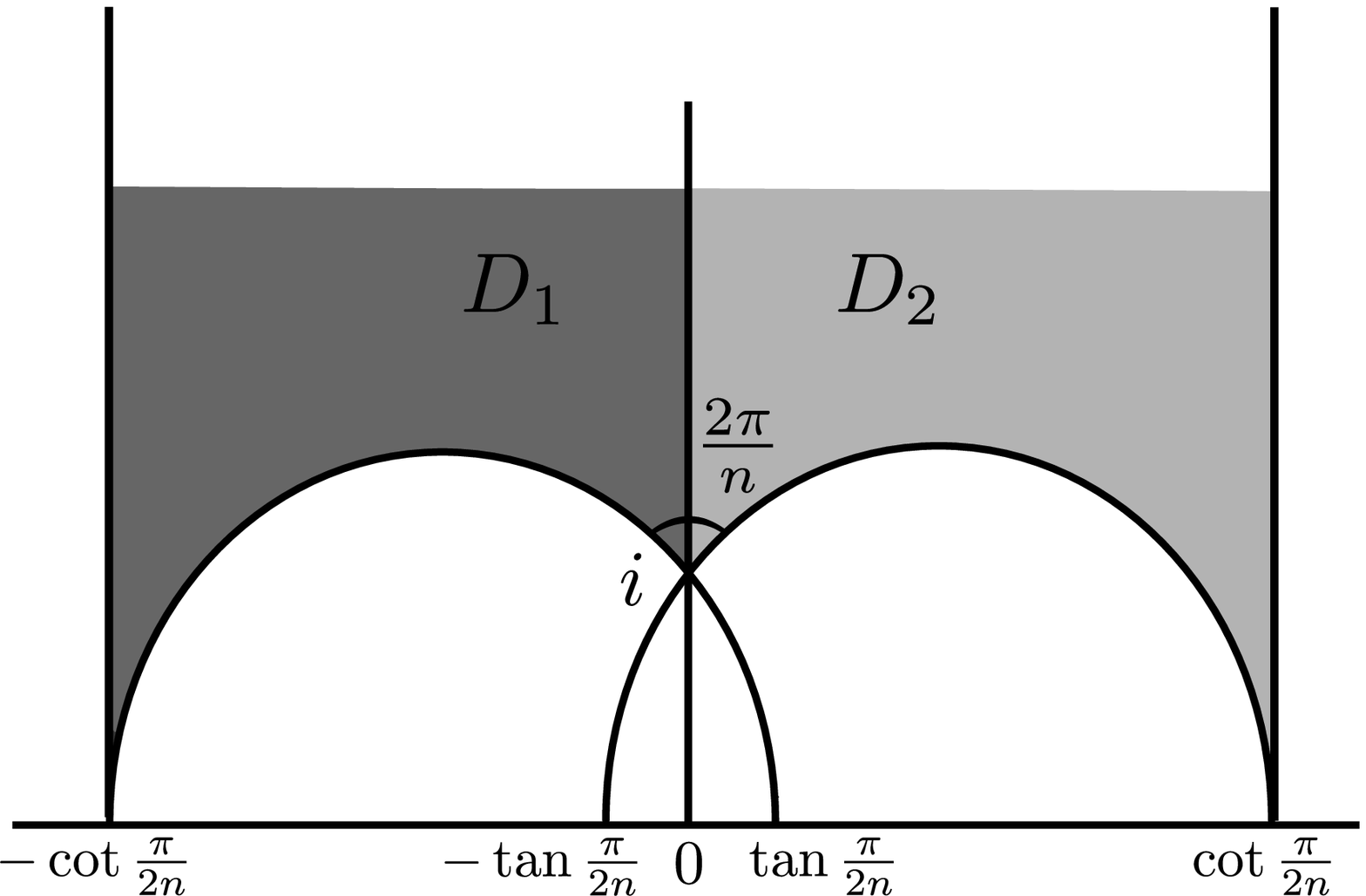}
\caption{}
\label{8}
 \end{center}
\end{figure}
\begin{example}
Let $X$ be the Riemann surface as Figure \ref{6}.
At the end of the calculation of $\Gamma(X)$, 
we have {\bf Gen}=$\{[T_8], [R_8T_8R_8^{-1}], [R_8^2]\}$ and {\bf Rep}=$\{[I], [R_8] \}$.
Since $[T_8]=[I\cdot T_8\cdot I^{-1}]$ is in $\Gamma(X)$, the $T$-side of $[I]$ and the $T^{-1}$-side of $[I]$ are identified by $\Gamma(X)$.
In the same way  the $T$-side of $[R_8]$ and the $T^{-1}$-side of $[R_8]$ are identified and 
 the $R$-side of $[R_8]$ and the $R^{-1}$-side of $[I]$ are identified. 
Hence $\mathbb{H}/\Gamma(X)$ has no genus, three punctures and one cone point with order  2 (see Figure \ref{9}).
\begin{figure}[h]
 \begin{center}
 \includegraphics*[keepaspectratio, scale=0.35]{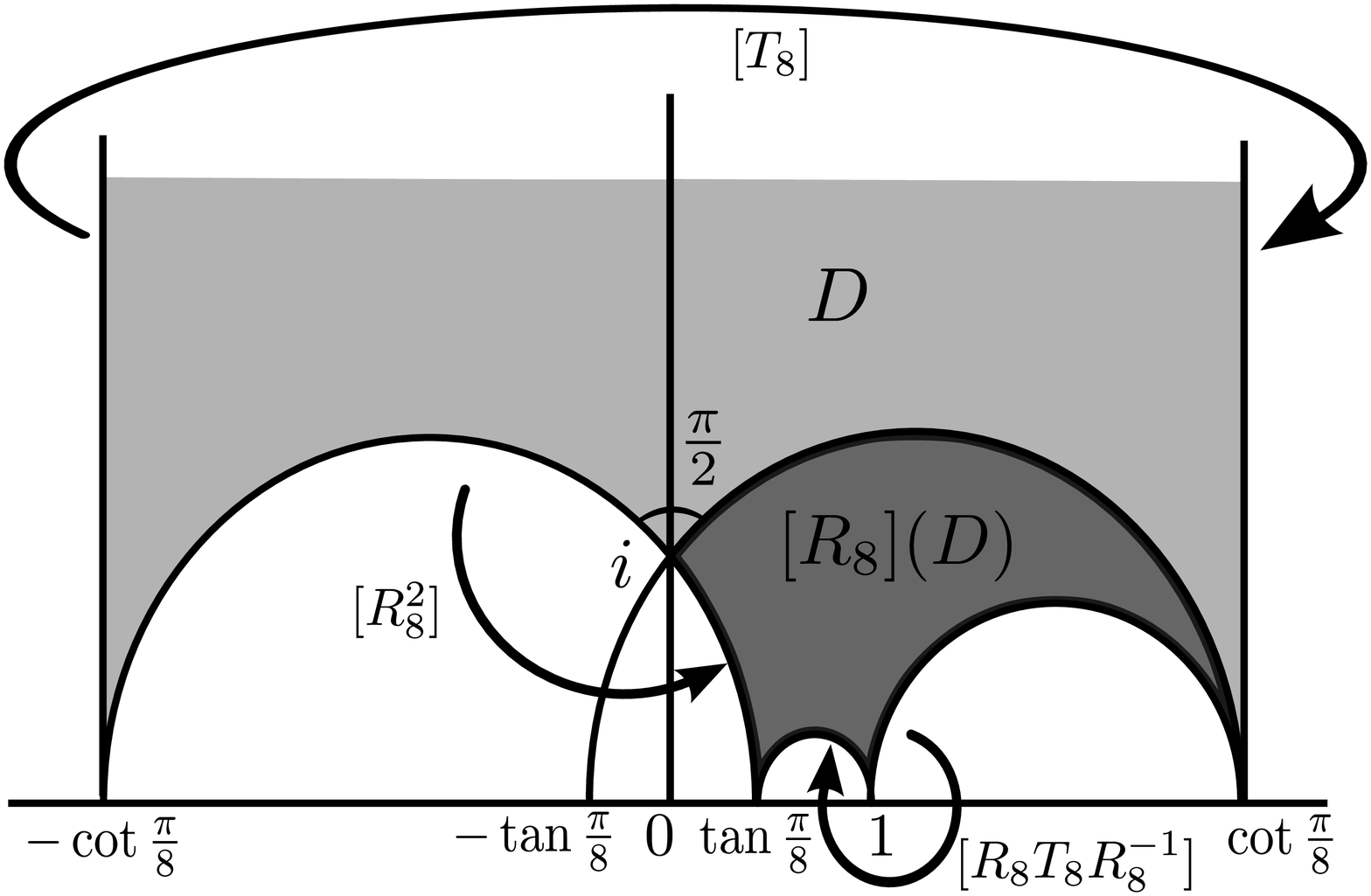}
\caption{}
\label{9}
 \end{center}
\end{figure}
\end{example}

\section{Veech groups of Abelian coverings} \label{Abel}
In this section, we show that the calculation of Veech group $\Gamma(X)$ by the Reidemeister-Schreier method always stops if $X$ is a finite Abelian covering of $P_{2n}$.
And we show that the calculations of
Veech groups of certain Abelian coverings can be done by using the corresponding
subgroups of ${\rm SL} (n, \mathbb{Z}_d)$.

Recall that if $\Gamma(X)$ is a finite index subgroup of $\left<[R_{2n}], [T_{2n}]\right>$, then the calculation of $\Gamma(X)$ stops by the proof of  Proposition \ref{R-S-method}.
We have a partial answer about the stop of calculations.

\begin{theorem} \label{main2}
Let $X$ be a finite Abelian covering of $P_{2n}$, that is, $X$ is a finite Galois covering of $P_{2n}$ and ${\rm Gal}(X/P_{2n})$ is an Abelian group.
Then the calculation of $\Gamma(X)$ stops.
\end{theorem}
\begin{proof}
Recall that $z_0$ is the point of $P_{2n}$ which corresponds to the center of the $2n$-gon $\Pi_{2n}$ as in Example \ref{exampleofveech2} and 
$\overline z_0 \in X$ is one of the preimages of $z_0$.
Since $X$ is a Galois covering, for each 
$w \in {\rm Gal}(\widetilde X_{2n}/P_{2n}) = \left<x_1, x_2, \cdot \cdot \cdot , x_{n}\right>$, 
$w$ is  in $\pi_1(X, \overline z_0) $ if and only if  $w$ is in $\pi_1(X, \overline z_1)$ for all $\overline z_1 \in X$.
Hence $[A]$ is in $\Gamma(X)$ if and only if $\gamma_A$ or $\gamma_{-A}$ fix $\pi_1(X, \overline z_0)$ for each$[A] \in \left<[R_{2n}], [T_{2n}]\right>$. 

As ${\rm Gal}(X/P_{2n})\cong \pi_1(P_{2n}, z_0)/ \pi_1(X, \overline z_0)$ is an Abelian group,
$\overline{x_ix_j}=\overline{x_jx_i}$ and $x_ix_j = x_j x_i\cdot  w$ for some $w \in \pi_1(X, \overline z_0)$.
Moreover,  set $d = {\rm lcm} \{{\rm ord}(\overline x_1), {\rm ord}(\overline x_2), \cdot \cdot \cdot , {\rm ord}(\overline x_{n})\}$ ,
then $x_i^d \in \pi_1(X,\overline z_1)$ for all $i$ and all $\overline z_1 \in X$.

Set $(e_1, e_2, \cdot \cdot \cdot , e_{2n})=I_{2n}$. We consider the homomorphism $\nu  :{\rm Gal}(\widetilde X_{2n}/P_{2n})\rightarrow \mathbb{Z}^n_d \ ; x_i \mapsto e_i$.
Then there exists a homomorphism $\Phi_d : \left<\gamma_T, \gamma_R \right> \rightarrow {\rm SL} (n,\mathbb{Z}_d)$ such that the following diagram is commutative. 
\[\xymatrix{
{\rm Gal}(\widetilde X_{2n}/P_{2n}) \ar[d]_\nu  \ar[r]^{\gamma_A}& {\rm Gal}(\widetilde X_{2n}/P_{2n}) \ar[d]^\nu\\
\mathbb{Z}^n_d \ar[r]^{\Phi_d(A)} &  \mathbb{Z}^n_d
}\]
Set $V = \nu (\pi_1(X, \overline z_0))$. 
For each $[A] \in \left<[R_{2n}], [T_{2n}] \right>$, if $[A]$ satisfies $\Phi_d(A)(V)=V$, then $[A]\in \Gamma(X)$.

Now we conclude that $\sharp$ {\bf Rep} $\leq  \sharp {\rm SL}(n,\mathbb{Z}_d)$  at every step of the calculation.
Suppose that $\sharp$ {\bf Rep} $> \sharp {\rm SL}(n,\mathbb{Z}_d)$ happens at some step of the calculation. 
Then there exists two distinct elements  $[A]$ and  $[B]$ in {\bf Rep} such that $\Phi_d(A)=\Phi_d(B)$.
Since $\Phi_d(AB^{-1})=I_{n}$ stabilizes $V$, \ $[A]\cdot [B]^{-1}$ is in $\Gamma(X)$.
However, since $[A]$ and $[B]$ are distinct elements in {\bf Rep},  $[A]\cdot [B]^{-1}$ is not in $\Gamma(X)$.
This is a contradiction.
\end{proof}
 From the proof of theorem \ref{main2}, we have the following.
 
\begin{corollary}\label{cor}
Let $X$ be a finite Abelian covering of $P_{2n}$. 
If there exists $d \in \mathbb{N}$ such that 
 $\{{\rm ord}(\overline x_1), {\rm ord}(\overline x_2), \cdot \cdot \cdot , {\rm ord}(\overline x_{n})\} = \{d\}$ or $\{1, d\}$,
then 
$[A]\in \Gamma(X)$ if and only if $\Phi_d(A)(V)=V$ for  each $[A] \in \left<[R_{2n}], [T_{2n}] \right>$.
\end{corollary}

\begin{example}
 Let $X$ be the covering of $P_8$ the same as Figure \ref{6}.
Then $X$ satisfies the assumption of Corollary \ref{cor}.
The fundamental group of $X$ is 
\begin{equation}
\nonumber
\pi_1(X,\overline z_0)=\left< x_1^2\ ,\ x_2\ ,\ x_4\ ,\ x_1x_3\ ,\ x_3x_1\ ,\ x_1^{-1}x_2x_1\ ,\ x_1^{-1}x_4x_1 \right>
\end{equation}
and 
\begin{equation}
\nonumber
V=\left< e_2\ ,\ e_4\ ,\ e_1+e_3 \right>_{\mathbb{Z}_2}.
\end{equation}
By Corollary \ref{cor}, for $[A] \in \left<[R_{8}], [T_{8}] \right>$, $[A]$ is in $\Gamma(X)$ if and only if  $\Phi_2(A)$ satisfies the followings : 
\begin{eqnarray}
\left\{
 \begin{array}{ll}
\nonumber 
\Phi_2(A)_{1,1}+ \Phi_2(A)_{3,1}+\Phi_2(A)_{1,3}+ \Phi_2(A)_{3,3}\equiv 0 \pmod 2\\
\Phi_2(A)_{1, j}+ \Phi_2(A)_{3, j}\equiv 0 \pmod 2 (j=2,4).
\end{array}
 \right.
 \end{eqnarray}
\end{example}

\section{Examples.}\label{Example}
Finally we show some examples of Veech groups that are calculated by the method of this paper.

\begin{example}
Let $X$ be the double covering of $P_8$ as Figure \ref{10}. 
Then $X$ is a Riemann surface of type $(3,2)$. 
Set $R=[R_8]$, $T=[T_8]$.
Then 
\begin{itemize}
\itemsep=8pt
\item For $[A] \in \left<R, T \right>$, $[A]$ is in $\Gamma(X)$ if and only if  $\Phi_2(A)_{1,j}\equiv 0 \pmod 2\  (j=2,3,4)$, 
\item $\Gamma(X) \\
=\Bigl< T, RT^2R^{-1}, RTRT^2(RTR)^{-1}, (RT)^3(RTRTR)^{-1}, $\\
$(RT)^2R^2T(RTR^2)^{-1}, (RT)^2R^3T(RTRTR^3)^{-1}, $\\
$RTR^2T(RTRTR^2)^{-1}, RTR^3T^2(RTR^3)^{-1},$ \\ $
RTR^3TR, R^2TR^{-2}, R^3(RTR^3T)^{-1} \Bigr>$, 

\item $\Gamma(X)\setminus \left< R, T \right> \\
=\left\{
   \begin{array}{c}
    {I, R, RT, R^2, RTR, RTRT, RTR^2, RTRTR,} \\
{RTRTR^2, RTRTR^3, RTR^3, RTR^3T} 
\end{array}
  \right\}$ and 
\item $\mathbb{H}/\Gamma(X)$ is a Riemann surface of type $(0,11)$.
\end{itemize}

\begin{figure}[h]
 \begin{center}
  \includegraphics*[keepaspectratio, scale=0.35]{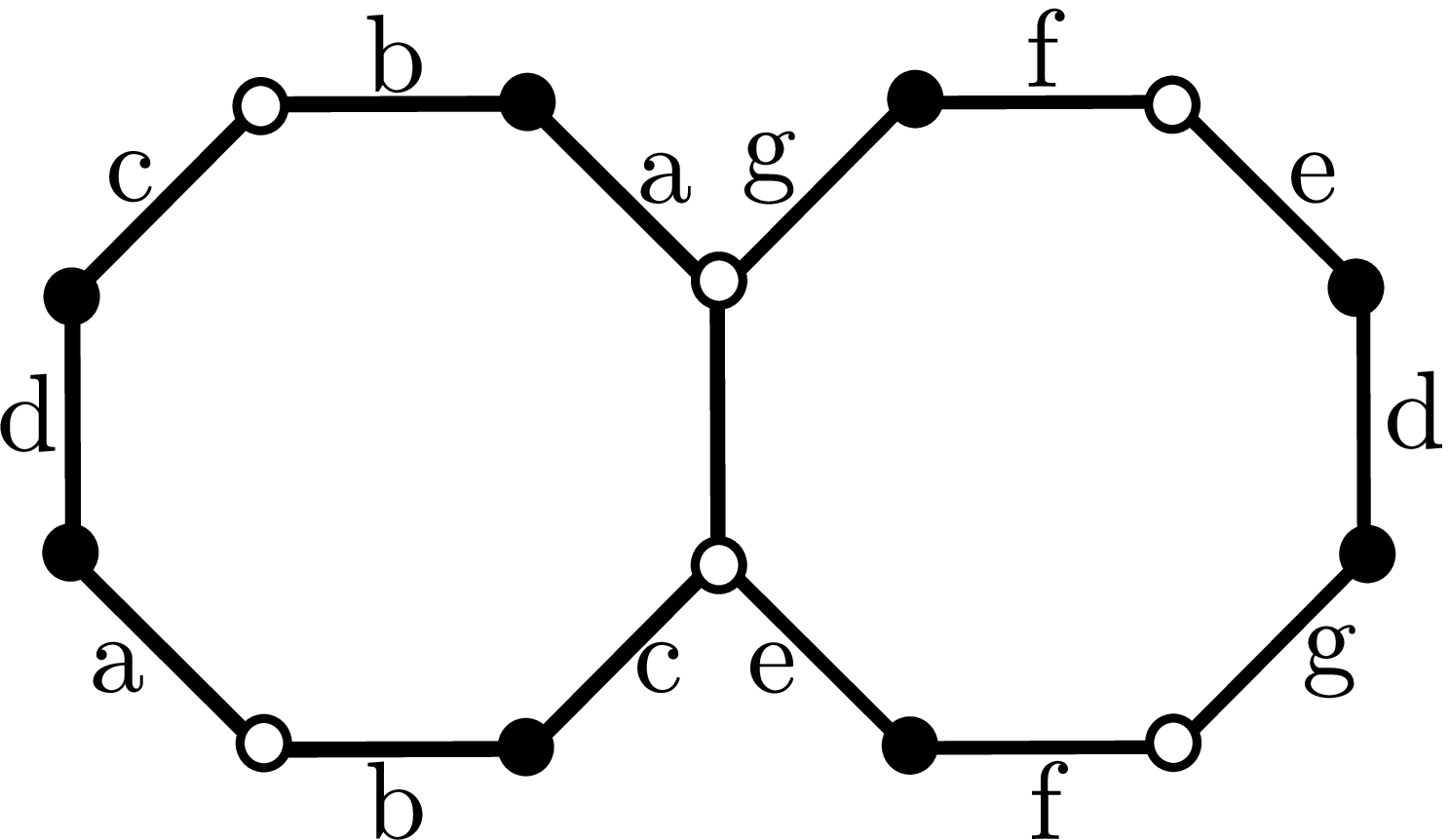}
  \caption{}
  \label{10}
 \end{center}
\end{figure}
\end{example}

\begin{example}
Let $X$ be the covering of $P_8$ as Figure \ref{11}. 
Then $X$ is a Riemann surface of type $(5,4)$. 
Set $R=[R_8]$, $T=[T_8]$.
Then
\begin{itemize}
\itemsep=8pt
\item For $[A] \in \left<R, T \right>$, $[A]$ is in $\Gamma(X)$ if and only if  $\Phi_4(A)$ satisfies the followings : 
\begin{eqnarray}
\left\{
 \begin{array}{ll}
\nonumber 
\displaystyle \sum_{i=1}^2 \left( \Phi_4(A)_{i,2} -\Phi_4(A)_{i,1}\right)
\equiv \sum_{i=3}^4 \left(\Phi_4(A)_{i,2} -\Phi_4(A)_{i,1} \right) \pmod 4,\\
\displaystyle \sum_{i=1}^2 \left( \Phi_4(A)_{i,1}+ \Phi_4(A)_{i,j} \right)
\equiv \sum_{i=3}^4 \left( \Phi_4(A)_{i,1}+ \Phi_4(A)_{i,j} \right) \pmod 4 (j=3,4),
\end{array}
 \right. 
 \end{eqnarray}
\item $\Gamma (X)= \left< T, R^2(RT)^{-1}, RT^2R^{-1}, RTRT(RTR)^{-1}, RTR^2 \right>$,
\item $\Gamma (X)\setminus \left< R, T \right> =\{ I, R, RT, RTR \}$ and
\item $\mathbb{H}/\Gamma(X)$ is a Riemann surface of type $(0,5)$.
\end{itemize}

\begin{figure}[h]
 \begin{center}
  \includegraphics*[keepaspectratio, scale=0.35]{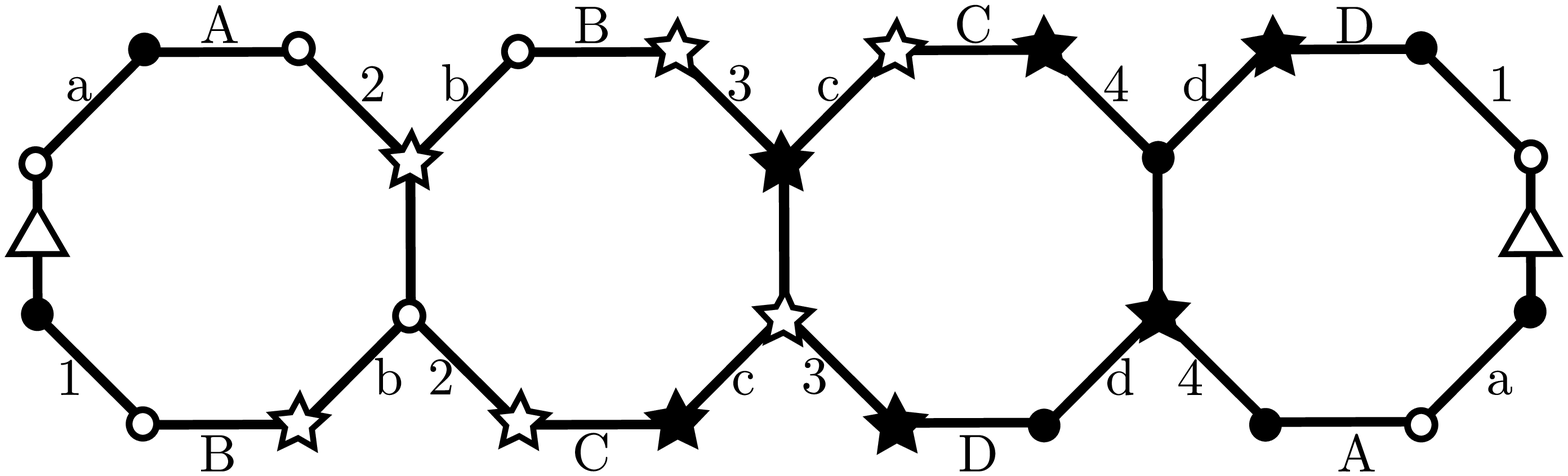}
  \caption{}
  \label{11}
 \end{center}
\end{figure} 
\end{example}

\begin{example}
$n\geq 2$. 
Let $X_{4n}$ be the double covering of $P_{4n}$ as Figure \ref{12}.
That is, $X_{4n}$ is constructed by gluing two regular $4n$-gons. 
Labels of small and capital letters appear in turn.
The sides whose labels are capital letters are identified with the opposite sides of another polygon
and others are identified with the opposite sides of the same polygon. 
Then $X_{4n}$ is a Riemann surface of type $(2n-1,2)$.
\begin{itemize}
\itemsep=8pt
\item For $[A] \in \left<[R_{4n}], [T_{4n}] \right>$, $[A]$ is in $\Gamma(X_{4n})$ if and only if  $\Phi_2(A)$ satisfies the followings : 
\begin{eqnarray}
\left\{
 \begin{array}{ll}
\nonumber 
\displaystyle \sum_{i=1}^{n} \Phi_2(A)_{2i-1,2j}\equiv 0 \pmod 2 (j=1, \cdot \cdot \cdot , n),\\
\displaystyle \sum_{i=1}^{n}\left ( \Phi_2(A)_{2i-1,1}+ \Phi_2(A)_{2i-1,2j-1}\right)
\equiv 0 \pmod 2 (j=2, \cdot \cdot \cdot , n),
\end{array}
 \right. 
 \end{eqnarray}
\item $\Gamma(X_{4n})= \left<[T_{4n}], [R_{4n}T_{4n}R_{4n}^{-1}],[R_{4n}^2] \right> $,
\item $\Gamma(X_{4n})\setminus \left< [R_{4n}],[T_{4n}] \right> = \{[I], [R_{4n}] \}$ and 
\item $\mathbb{H}/\Gamma(X_{4n})$ is an orbifold which has no genus, 3 punctures and one cone point whose order is $n$.
\end{itemize}

\begin{figure}[h]
 \begin{center}
  \includegraphics*[keepaspectratio, scale=0.33]{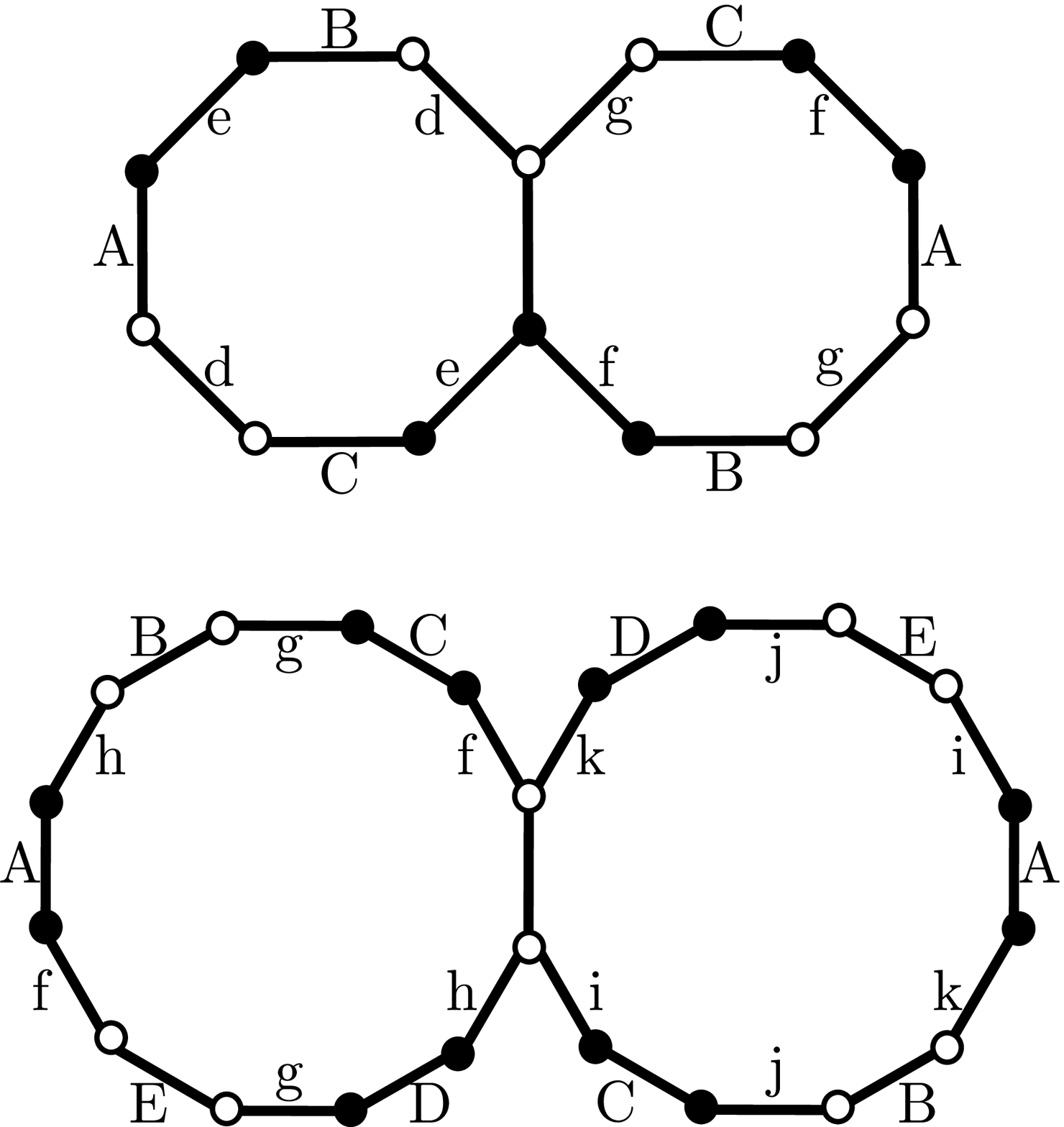}
  \caption{}
  \label{12}
 \end{center}
\end{figure} 

\end{example}

\begin{example}
For each $n\geq 2$, let $X_{4n}$ be the double covering of $P_{4n}$ as Figure \ref{13}.
That is, horizontal and vertical sides of two polygons are identified with the opposite sides of another polygon 
and others are identified with the opposite sides of the same polygon. 
Then $X_{4n}$ is a Riemann surface of type $(2n-1,2)$. 
\begin{itemize}
\itemsep=8pt
\item For $[A] \in \left<[R_{4n}], [T_{4n}] \right>$, $[A]$ is in $\Gamma(X_{4n})$ if and only if  $\Phi_2(A)$ satisfies the followings : 
\begin{eqnarray}
\left\{
 \begin{array}{ll}
\nonumber 
 \Phi_2(A)_{1,1} + \Phi_2(A)_{n+1,1} + \Phi_2(A)_{n+1,1} + \Phi_2(A)_{n+1,n+1} \equiv 0 \pmod 2,\\
\Phi_2(A)_{1,j} + \Phi_2(A)_{n+1,j} \equiv 0 \pmod 2 (j=2, \cdot \cdot \cdot , n, n+2, \cdot \cdot \cdot , 2n),
\end{array}
 \right. 
 \end{eqnarray}
\item $\Gamma(X_{4n})= \left<[R_{4n}^{i}T_{4n}R_{4n}^{-i}],[R_{4n}^n] \mid  i = 0, 1, \cdot \cdot \cdot, n-1 \right> $,
\item $\Gamma(X_{4n})\setminus \left< [R_{4n}],[T_{4n}] \right> = \{[I], [R_{4n}], [R_{4n}^2], \cdot \cdot \cdot , [R_{4n}^{n-1}] \}$ and
\item $\mathbb{H}/\Gamma(X_{4n})$ is an orbifold which has no genus, $2n+1$ punctures and one cone point whose order is 2.
\end{itemize}

\begin{figure}[h]
 \begin{center}
  \includegraphics*[keepaspectratio, scale=0.33]{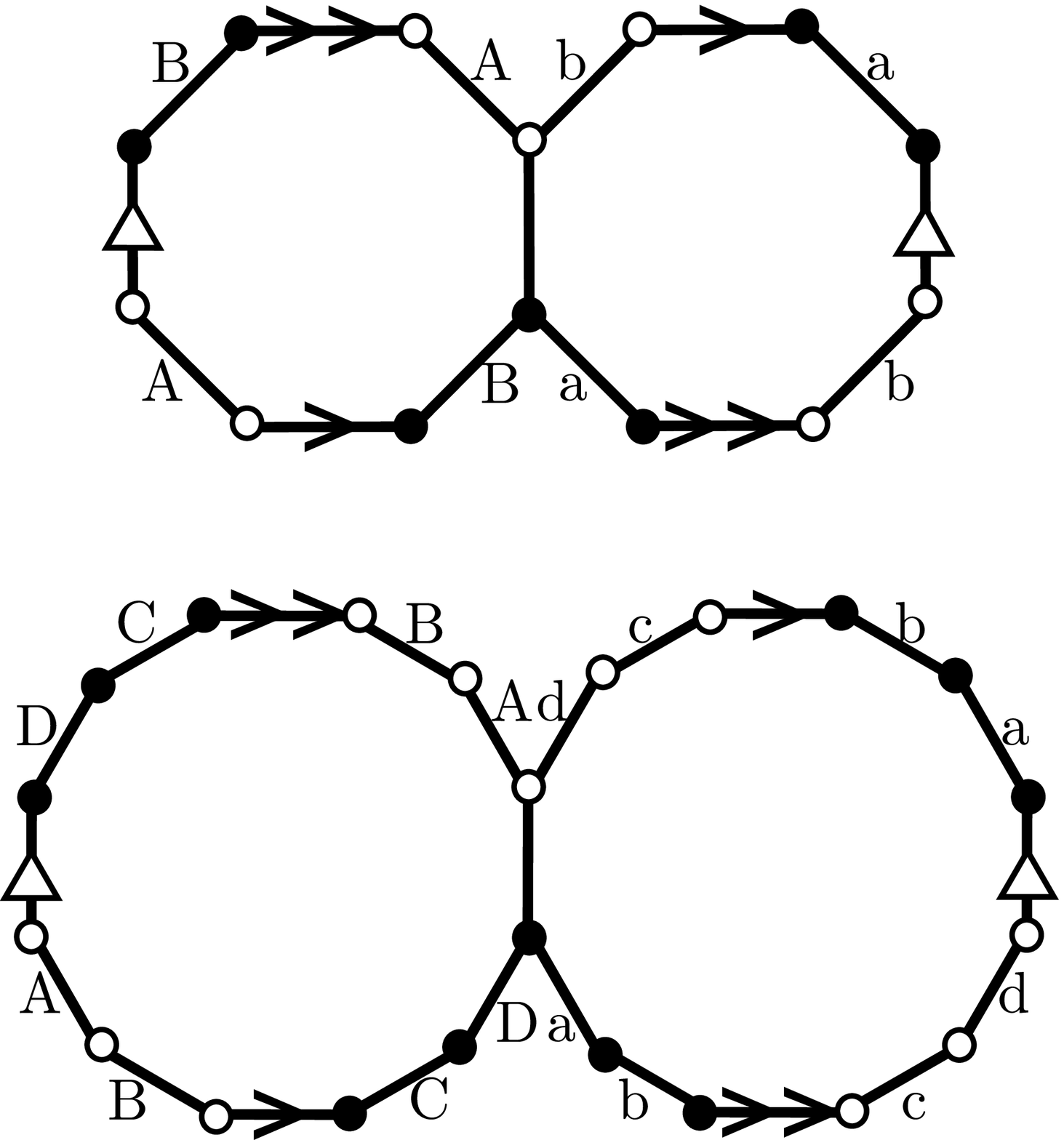}
  \caption{}
  \label{13}
 \end{center}
\end{figure} 
\end{example}

\newpage

\begin{example}
Let $X_d$ be the covering of $P_8$ with degree $d$ as Figure \ref{14}.
Then $X_d$ is a Riemann surface of type $(d+1, d)$.
And , for $[A] \in \left<[R_{8}], [T_{8}] \right>$, $[A]$ is in $\Gamma(X_d)$ if and only if  $\Phi_d(A)_{1,j}\equiv 0 \pmod d (j=2,3,4)$. 
The next is a chart about Veech groups $\Gamma(X_d)$. 
Here,
\begin{itemize}
\item $\sharp$ Rep is the index of $\Gamma(X_d)$ in $\left<[R_8], [T_8] \right>$,
\item $\sharp$ Gen is a number of generators of $\Gamma(X_d)$ by this calculation,
\item ``genus'' is the genus of  $\mathbb{H}/\Gamma(X_d)$,
\item ``puncture'' is the number of punctures of  $\mathbb{H}/\Gamma(X_d)$ and 
\item ``cone point (order)'' is the number of cone points of  $\mathbb{H}/\Gamma(X_d)$ and their orders.
\end{itemize}

\begin{figure}[h]
 \begin{center}
  \includegraphics*[keepaspectratio, scale=0.88]{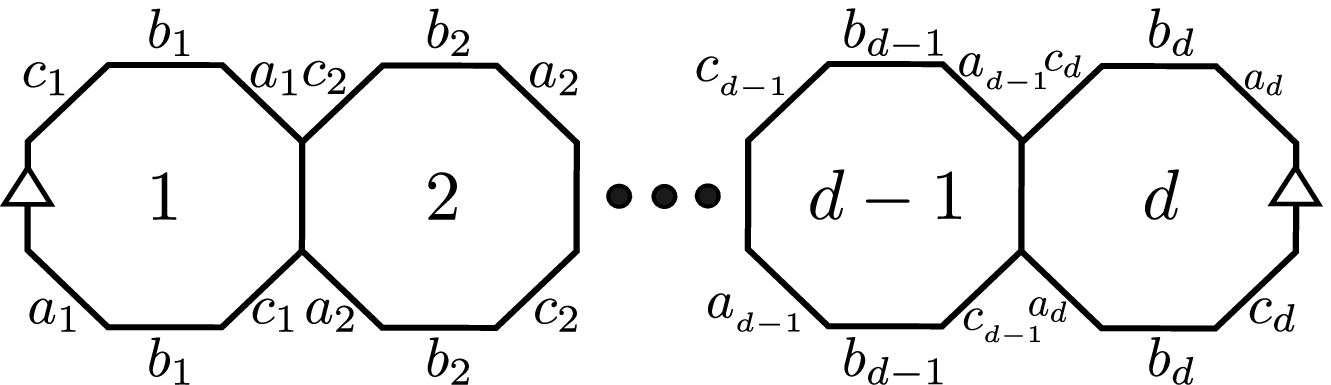}
  \caption{}
  \label{14}
 \end{center}
\end{figure} 

\begin{center}
\begin{tabular}{|r|c|c|c|c|c|c}
\hline
 d & $\sharp$ Gen & $\sharp$ Rep  &  genus & puncture & cone point (order) \\ \hline \hline
 2 &   11  & 12   &   0    &   11   & 0  \\
 3 &   29  & 32  &  1    &   24   & 0  \\ 
 4 &   87  &  96  &  8    &   58   & 0  \\ 
 5 &   142  &  156 & 24    &  68  & 6  (2,2,2,2,2,2)\\
 6 &   349  &  384 &   45   &   200  & 0  \\  
 7 &   367  &  400 &  87   &   128  & 0  \\ 
 8 &   704  &  768 &  149   &   280  & 0  \\ 
 9 &   785  &  864 & 185  &   280  & 0  \\
10 &  1704 &  1872 & 419  &  568   & 0  \\
11 &  1353 &  1464 &  400  &  300   & 0  \\
\hline
\end{tabular} 
\end{center}
\end{example}
\ \\
\begin{example}
Let $X_d$ be the covering of $P_8$ with degree $d$ as Figure \ref{15}.
Then $X_d$ is a Riemann surface of type $(d+1, d)$.
And, for $[A] \in \left<[R_{8}], [T_{8}] \right>$, $[A]$ is in $\Gamma(X_d)$ if and only if  $\displaystyle \sum_{i=1}^{4} \left( \Phi_d(A)_{i,j}- \Phi_d(A)_{i,1} \right) \equiv 0 \pmod d (j=2, 3, 4)$. 
The next is a chart about Veech groups $\Gamma(X_d)$. 
\begin{figure}[h]
 \begin{center}
  \includegraphics*[keepaspectratio, scale=0.88]{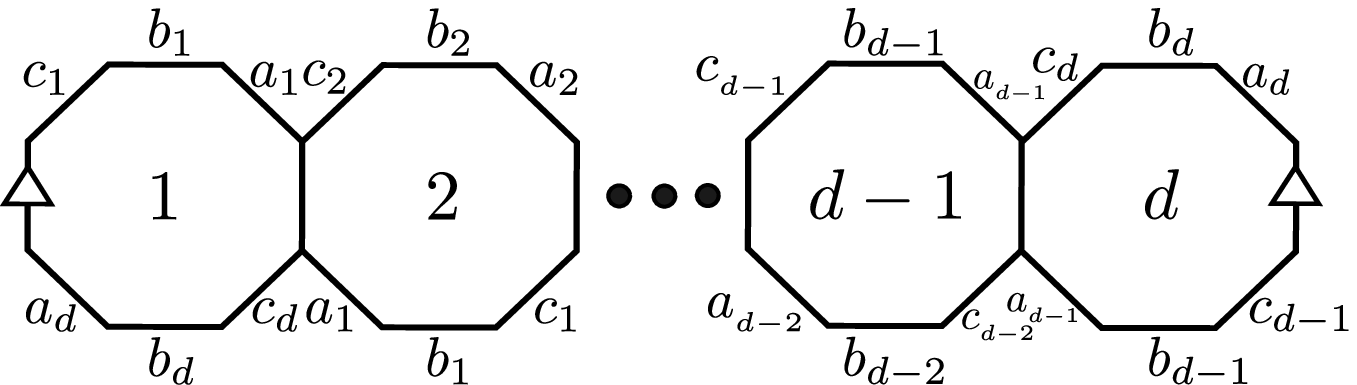}
  \caption{}
  \label{15}
 \end{center}
\end{figure} 
\begin{center}
\begin{tabular}{|r|c|c|c|c|c|c}
\hline
 d & $\sharp$ Gen & $\sharp$ Rep  &  genus & puncture & cone point (order) \\ \hline \hline
 2 &    2    &  1   & 0    &   2   & 1 (4) \\
 3 &    29   & 32  &  1    &   24  & 0  \\ 
 4 &    5    & 4    &  0    &   5   & 0  \\ 
 5 &    142 & 156 &  24   &   68  & 6  (2,2,2,2,2,2)\\
 6 &    29   &  32  &  1    &   24  & 0  \\  
 7 &    367 & 400  &  87   &   128 & 0  \\ 
 8 &    29   & 32   &  1    &   24  & 0  \\ 
 9 &    789 &  864 &  185  &   280 & 0  \\
10 &    142 &  156 & 24   &   68  & 6  (2,2,2,2,2,2)\\
11 &    1353 & 1464&  400  &   300 & 0  \\
12 &    115  &  128 &  11   &   76  & 0  \\
13 &    2220 & 2380 & 682   &   416 & 14 (2,2,2,2,2,2,2,2,2,2,2,2,2,2)\\
14 &    367  & 400 & 87    &   128 & 0 \\
\hline
\end{tabular} 
\end{center}
\end{example}

\section*{Acknowledgments}
This work was supported by Global COE Program ``Computationism as a Foundation for the Sciences".
The author thanks Professor Hiroshige Shiga for his valuable suggestions and comments.

\bibliography{ref}

\end{document}